\documentclass{article}
\usepackage{graphicx} 
\usepackage{amsmath,amssymb,amsthm,mathtools}
\usepackage{hyperref} 
\usepackage{mathrsfs}
\usepackage{bbm}

\usepackage{xcolor}
\usepackage{listings}

\lstdefinelanguage{Sage}{
  language=Python,
  morekeywords={PolynomialRing,QQ,FractionField,Matrix}
}

\newcommand{\R}{\mathbb{R}}
\def\bZ{\mathbb Z}

\newcommand{\T}{\mathbb{T}}
\newcommand{\RP}{\mathbb{R}\mathrm{P}}
\newcommand{\Free}{\operatorname{Free}}
\newcommand{\Spin}{\operatorname{Spin}}
\newcommand{\D}{\mathcal{D}}
\def\cD{\D}
\newcommand{\dd}{\partial}

\theoremstyle{plain}
\newtheorem{theorem}{Theorem}
\newtheorem{proposition}{Proposition}
\newtheorem{lemma}{Lemma}
\newtheorem{conjecture}{Conjecture}
\newtheorem{corollary}{Corollary}

\def\bS{\mathbb S}
\def\bT{\mathbb T}
\def\bR{\mathbb R}
\newcommand{\BF}{\bf\boldmath }

\newcommand{\Z}{\mathbb{Z}}
\newcommand{\SO}{\operatorname{SO}}

\newcommand{\GL}{\operatorname{GL}}

\DeclareMathOperator{\sgn}{sign}

\title{Free maps in critical dimension on low-dimensional tori and closed surfaces
}

\author{Roberto De Leo\\ \small\em roberto.deleo@howard.edu\\
\small\em 
Department of Mathematics,
Howard University\\ \small\em 
Washington, DC 20059 (USA)\\
}
\date{\today}
\begin{document}

\maketitle

\begin{abstract}
    We present a method to build free immersions in critical dimension on $m$-tori for $m = 2, 3, 4, 5$ by using a factorization trick inspired by tori immersions in critical dimension. 
    As an application, we show that the set of smooth free maps from a closed surface $M$ to $\bR^5$ is nonempty. In particular, every closed surface embeds freely in $\bR^5$.
    Finally, we find the homotopy invariants of free maps in our environment and study which invariants are achieved by our maps.
\end{abstract}



\section{Introduction}

It was shown first by John Nash in \cite{Nas56} that free maps of manifolds into Euclidean spaces play an important role in the theory of isometric immersions (e.g. see \cite{Gro86} for several fundamental results involving free maps).

\medskip\noindent
{\bf Free maps.}
Given a smooth \(m\)-dimensional manifold \(M\), set 
$$
\text{\BF$q_m$}=\frac{m(m+3)}{2}.
$$
We say that a map \(f\in C^2(M,\R^q)\) is {\bf free} if, in local coordinates \((x^\alpha)\), \(\alpha=1,\ldots,m\), on \(M\) and \((y^i)\), \(i=1,\ldots,q\), on \(\R^q\), $q\geq q_m$, the {\bf osculating matrix}
\[
\text{\BF$\D f$}=(\dd_\alpha f^i,\dd_{\alpha\beta}f^i)
\]
of the first and second derivatives of \(f\) has constant rank equal to $q_m$.
Notice that every free map is an immersion.
We denote by {\BF$\Free^k(M,\bR^q)$} the set of all free $C^k$ immersions of $M$ into $\bR^q$.

\medskip\noindent
{\bf Freeness and affine connections.}
The condition of being free is not intrinsic.
Recall indeed that, close to any point where the Jacobian of $f$ has constant rank $k$, we can find local coordinates $(x^\alpha)$ on $M$ and \((y^i)\) on $\bR^q$ where $f$ writes as the projection $(x^\alpha)\mapsto(x^1,\dots,x^k,0,\dots,0)$, which is clearly not free.
This is due to the fact that, given a second pair of charts $(x^{\alpha'})$, \((y^{i'})\), the coordinate transformation of the $\partial^2_{\alpha'\beta'}f^{i'}$ are {\em affine}, rather than {\em linear}, in the $\partial_{\alpha\beta}f^{i}$:
\begin{align*}
    \partial_{\alpha'}f^{i'}&=\partial_i y^{i'}\partial_{\alpha}f^i\partial_{\alpha'}x^\alpha,\\
    \partial_{\alpha'\beta'}f^{i'}&=\partial_i y^{i'}(\partial_{\alpha\beta}f^i\partial_{\alpha'}x^\alpha\partial_{\beta'}x^\beta+\partial_{\alpha}f^i\partial_{\alpha'\beta'}x^\alpha)+\partial_{ij}y^{i'}\partial_{\alpha}f^i\partial_{\alpha'}x^\alpha\partial_\beta f^j\partial_{\beta'}x^\beta.
\end{align*}

In order to have a combination of second derivatives that is {\em tensorial} is enough setting an affine connection $\Gamma$ on $T\bR^q$ and an affine connection $\hat\Gamma$ on $TM$.
Indeed, in this case, the differential $df$ is a section of the bundle $T^*M\otimes f^*(TN)$ and one can define the Hessian as the covariant derivative of the differential:
$$
(\nabla df)^i_{\alpha\beta} 
= 
\partial_{\alpha\beta}f^i-\hat\Gamma^\gamma_{\beta\alpha}\partial_\gamma f^i+\Gamma^i_{jk}\partial_\alpha f^j \partial_\beta f^k.
$$
Then, given any other pair of local charts $(x^{\alpha'})$, \((y^{i'})\), we have that
$$
(\nabla df)^{i'}_{\alpha'\beta'} 
= \partial_i y^{i'}
(\nabla df)^i_{\alpha\beta} \partial_{\alpha'}x^\alpha\partial_{\beta'}x^\beta.
$$
Hence, freeness is an intrinsic concept within the class of manifolds endowed with an affine connection on their tangent bundle; in particular, it is an intrinsic concept within the class of Riemannian manifolds.

\medskip\noindent
{\bf The affine connection on $M$ does not play a significant role.}
Other than showing that freeness is intrinsic, the connection on $M$ can be disregarded.
The reason is that freeness is about the span of the first and second derivatives and a change of coordinates in $M$ would leave that span unchanged. 
Since every manifold has a Riemannian metric, an affine connection exists on every manifold and so there are no obstructions on this side and we can ultimately forget about this connection.

\medskip
From a concrete viewpoint, the considerations above show that talking naively about free maps $M\to\bR^q$ makes sense as long as we allow only affine changes on $\bR^q$.
We make this assumption throughout the rest of the article.


\medskip\noindent
{\bf Existence of free maps.}
Within their work on elimination of singularities, it was shown by Yakov Eliashberg and Misha Gromov~\cite{EG71} (see also~\cite[Section 1.1.4]{Gro86}) that the partial differential relation of being a free map \(M\to \R^q\) satisfies the \(h\)-principle if either \(q>q_m\) or if \(q=q_m\) and \(M\) is open (i.e. has no compact connected component). 
In particular this means that, under those hypotheses, a sufficient condition for free maps \(M\to \R^q\) to arise is that \(M\) be parallelizable.

\medskip\noindent
{\BF A canonical free map $\bR^m\to\bR^{q_m}$}
Since $\bR^m$ is open, it has free maps into $\bR^{q_m}$.
In this case, we can easily write explicit ones.
Let \((x^\alpha)\), \(\alpha=1,\ldots,m\), be linear coordinates on \(\R^m\). 
An elementary smooth free map \(F:\R^m\to \R^{q_m}\) is given by any \(q_m\)-ple of all monomials of first and second degree in the \(x^\alpha\), e.g.
\[
F(x^1,\ldots,x^m)=\bigl(x^1,\ldots,x^m,(x^1)^2/2,\ldots,(x^m)^2/2,x^1x^2,\ldots,x^{m-1}x^m\bigr).
\]
For a convenient ordering of the derivatives we get that
\[
\D F=\begin{pmatrix}
{\mathbbm 1}_m & *\\
O_{r_m} & {\mathbbm1}_{r_m}
\end{pmatrix},
\]
where $r_m=m(m+1)/2$, so that \(\det \D F=1\). Clearly every map \(G:\R^m\to \R^q\), \(q>q_m\), which reduces to \(F\) after some projection of \(\R^q\) into \(\R^{q_m}\) is free as well.

\medskip\noindent
{\bf Free maps in critical dimension on compact manifolds.}
It was pointed out by Eliashberg and Mishachev in~\cite[Intrigue]{EM02} and by Gromov in \cite[Appendix 6]{GR70} (see also~\cite[Section 1.1.4]{Gro86} and~\cite[Section~1.2, case 5]{Gro17}), 
that it is unknown whether free maps arise in the ``critical dimension'' case (i.e. for \(q=q_m\), so that \(\D f\) is a square matrix) when \(M\) is compact, besides the well-known cases of spheres \(\bS^m\) and projective spaces \(\RP^m\).

\medskip\noindent
{\bf Free maps in critical dimension on $\RP^m$ and $\bS^m$.}
Consider the free embedding \(F\) of \(\R^m\) into \(\R^{q_m}\) shown in Example 1, homogenize its components by introducing a new coordinate \(x^{m+1}\) and add to it a further component \((x^{m+1})^2\). 
The resulting map
\[
f(x^1,\ldots,x^{m+1})=\bigl(x^1x^{m+1},\ldots,x^mx^{m+1},(x^{m+1})^2,(x^1)^2,\ldots,x^{m-1}x^m\bigr)
\]
sends homogeneously \(\R^{m+1}\) into \(\R^{q_m+1}\) and induces canonically the so-called (quadratic) Veronese map $\nu_2:\RP^m\longrightarrow \RP^{q_m}$ given by 
\[
\nu_2([x^1:\cdots:x^{m+1}])= [x^1x^{m+1}:\cdots:x^ix^j:\dots:x^{m-1}x^m]_{1\leq i\leq j\leq m+1}.
\]
Denote by $Y^k$ the homogeneous coordinate such that 
$$
Y^k(\nu_2([x^1:\cdots:x^{m+1}]))=(x^k)^2,\;1\leq k\leq m+1.
$$
Then, since at least one of the coordinates in $[x^1:\cdots:x^{m+1}]$ is non-zero, the image $\nu_2(\RP^m)$ does not intersect the set
$$
Y^1+\dots+Y^{m+1}=0,
$$
that is the set of points at infinity of the $q_m$-dimensional plane represented by the affine chart
$$
Y^1+\dots+Y^{m+1}\neq0.
$$
Hence, one can think of $\nu_2$ as a map to $\bR^{q_m}$.

In the affine chart $x^{m+1}=1$, this map coincides with the canonical free map from $\bR^m$ to $\bR^{q_m}$ and similarly happens on any other chart $x^k=1$. 
Since these charts cover every point of $\RP^{m}$, it follows that $\nu_2$ is free.



Observe finally that \(\nu_2\) extends uniquely to an immersion \(\hat\nu_2:S^m\to \R^{q_m}\) invariant under the antipodal map. Then also \(\hat\nu_2\in\Free(\bS^m,\R^{q_m})\) since locally \(\hat\nu_2\) and \(\nu_2\) are the same function.

\medskip\noindent
{\bf Free maps on tori.}
After spheres and projective spaces, the simplest compact manifolds are arguably the tori \(\T^m\). 
Since tori are parallelizable, the results of Eliashberg and Gromov grant the existence of free immersions 
$$
\T^m\to \R^{q_m+1}
$$
but says nothing about free immersions in critical dimension.

The 1-torus happens to coincide the 1-sphere and so, because of the example above, we know that it admits free maps in critical dimension. 
A free map \(\bS^1\to \R^q\) is just a loop whose radius of curvature is everywhere finite, so that its velocity and acceleration are always transversal.
The Veronese map in this case is equivalent to the canonical embedding of the unit circle on the plane, namely the map
\[
s(x)=(\cos x,\sin x).
\]
Indeed
\[
\D s=\begin{pmatrix}
-\sin x & \phantom{-}\cos x\\
-\cos x & -\sin x
\end{pmatrix},
\]
so that \(\det\D s=1\). 

An elementary family of free maps can be built using the embedding $s(x)$ by setting
$$
f_m(x^1,\dots,x^m) =
(s(x^i),s(x^i+x^j))_{1\leq i\leq m,1\leq j<i\leq m}
$$
This is a free map $\bT^m\to\bR^{m(m+1)}$.
Indeed, in case $m=2$, the map is 
\[
f_2(x,y)=(s(x),s(y),s(x+y))
\]
and a direct calculation, after ordering the rows of \(\D f_2\) as \((\dd_x,\dd_{xx},\dd_y,\dd_{yy},\dd_{xy})\), shows that
\[
\D f_2=\begin{pmatrix}
\D s(x)&O_2&*&*\\
O_2&\D s(y)&*&*\\
0\ 0&0\ 0&-\cos(x+y)&-\sin(x+y)
\end{pmatrix},
\]
which clearly has full rank at every point.
The case $m>2$ can be treated in a similar way.

\medskip
To the knowledge of the author, explicit expressions for free maps $\bT^m\to\bR^q$ with $q<m(m+1)$ are not available in literature.
Moreover, as pointed out above, the  existence of free maps in critical dimension on $\bT^n$ is not established in literature.

\medskip\noindent
{\bf Main results and structure of the article.}
In Section~\ref{sec: tori} we show that $\Free^\infty(\bT^m,\bR^{q_m})\neq\varnothing$ for $m=2,3,4,5$ by presenting a method to generate explicit expressions of free maps in critical dimension on $\bT^m$ for such values of $m$.
As a further application of this method, we also show that $\bT^2$ admits ``k-free'' smooth immersions, namely immersions such that its vectors of partial derivatives from order 1 up to order $k$ are linearly independent at every point, in critical dimension for $k=3,4,5,6$.

Then, in Section~\ref{sec: surfaces}, we show that the existence of free maps in critical dimension on closed surfaces $M$ and $N$ implies the existence of similar free maps on their connected sum $M\# N$ and we use the existence of free maps on $\RP^2$ and $\bT^2$ to conclude that free maps arise on any closed surface.
Finally, in Section~\ref{sec: invariants}, we show that ``formal free maps'', namely smooth sections of the bundle ${\cal R}_m=\bT^m\times\bR^m\times GL_{q_m}(\bR)\to\bT^m$, are characterized, for $2\leq m\leq 5$, by three homotopic invariants and study which of these invariants are achieved by free maps generated through our method.
In particular, we show that the map $\pi_0(\Free^\infty(\bT^2,\bR^5))\to \pi_0(\Gamma^\infty{\cal R}_m)$ is surjective.

\section{Free maps in critical dimension on tori}
\label{sec: tori}
The idea behind the results of this section comes from the following elementary observation about immersions (recall that $q=m+1$ is the critical dimension for the existence of immersions $\bT^m\to\bR^q$): 

\medskip
{\em for every $m=2,3,\dots$ there are immersions $\bT^m\to\bR^{m+1}$ that are obtained as the image of an immersed loop $v:\bS^1\to\bR^{m+1}$ under an action of $\bT^{m-1}\to SO(m+1)$ by successive rotations in suitable coordinate 2-planes.}

\medskip
The fundamental block of these isometric toric actions is the elementary action of $\bS^1$ on the plane given by the matrices
$$
\text{\BF$A(x)$}=
\begin{pmatrix}
\cos x & -\sin x \\
\sin x & \phantom{-}\cos x
\end{pmatrix}.
$$
For instance, the canonical embedding of $\bT^2$ in $\bR^3$ is obtained as the image of the loop
$$
v(t)=(R+r\cos t,0,r\sin t)
$$
under the action of $\bS^1$ on $\bR^3$ given by
$$
A_3(x) = \begin{pmatrix}
A(x) & 0\\
0 & 1
\end{pmatrix}.
$$
It is natural therefore to pose the following question:

\medskip
{\em are there integers $m$ for which there is a loop $\gamma:\bS^1\to\bR^{q_m}$ and a torus action $\bT^{m-1}\to SO_{q_m}$ such that the image of $\gamma$ under the action of $\bT^{m-1}$ is a free immersion of $\bT^m$ into $\bR^{q_m}$?}
\medskip

Notice that the matrix $A(x)$, which is the building block of the representations of the toric actions we are considering,  coincides, modulo an orientation switch, with the osculating matrix of the map $s(x)=(\sin x,\cos x)$.

\medskip
The examples exhibited below show that the answer to this question is positive at least for $m=2,3,4,5$.
The method we use, in short, is looking for free maps $F:\bT^m\to\bR^{q_m}$ that factor as a $q_m\times q_m$ matrix $R$, depending only on the first $m-1$ variables, acting on a vector $v\in\bR^{q_m}$, depending only on the last variable.
A key point of this factorization is that the determinant of the osculating matrix of the corresponding map $F$ depends only on the last coordinate, making it easy to verify whether the determinant has zeros.

Our explorations of the cases $m=6$ and $m=7$ were unable to find any free map in critical dimension.
This is possibly just a sign that the size of the osculating matrix for $m>5$ is too large for this method to be effective.
The simplest next step would be to find some pattern in free maps in critical dimension on low-dimensional tori that could be used to prove by induction the existence of free maps in critical dimension on $\bT^n$ for every $n$.



\subsection{A free map $\bT^2\to\bR^5$}
\label{sec: T2}
We use coordinates $(x,y)\in\bT^2$.
Define
\[
F_2:\T^2\longrightarrow \mathbb R^5
\]
by
\[
F_2(x,y)=R_2(x)v_2(y),
\]
where
\[
R_2(x)=
\begin{pmatrix}
A(x)&O_2&0&\\
O_2&A(2x)&0&\\
0\ 0&0\ 0&1
\end{pmatrix}
\]
and $v_2(y)$ is an immersion of $\bS^1$ into $\bR^5$ to be determined.

For short, we write 
$
R_2(x) = \operatorname{diag}(A(x),A(2x),1)\in SO(5).
$
Let now 
\[
J=
\begin{pmatrix}
0&-1\\
1&0
\end{pmatrix}
\]
and
\[
X=\operatorname{diag}(J,2J,0).
\]
Then
\[
R'_2(x)=R_2(x)X.
\]
Therefore
\[
\partial_x F_2=R_2Xv_2,\qquad \partial_y F_2=R_2v_2',
\]
\[
\partial_{xx}F_{2}=R_2X^2v_2,\qquad \partial_{xy}F_{2}=R_2Xv_2',
\qquad \partial_{yy}F_{2}=R_2v_2''.
\]
Set 
$$
D(y) = \det(Xv_2(y),v'_2(y),X^2v_2(y),Xv'_2(y),v''_2(y)).
$$
Since \(R_2(x)\in SO(5)\), we get
\[
\det\D F_2
=
\det R_2(x) D(y) = D(y).
\]
Thus, the freeness determinant depends only on \(y\) and this type of map $F_2$ is free if and only if the five vectors
$$
v_2',v_2'',Xv_2,Xv_2',X^2v_2
$$
are linearly independent at every point.
A computer-aided search led to the following simple expression:
\[
v_2(y)
=
\begin{pmatrix}
1+\sin y\\
1\\
-2-\sin y\\
1+\cos y + \sin y\\
\cos y - 3\sin y
\end{pmatrix}.
\]

A direct computation gives
\[
D(y)
=
6+2\cos(2y)+\sin(2y).
\]
Since
$$
6+2\cos(2y)+\sin(2y)\geq 6-2-2>0,
$$
it means that 
\[
\det\D F_2\neq0
\]
everywhere on \(\bT^2\) and therefore
\[
F_2\in \operatorname{Free}^{\infty}(T^2,\mathbb R^5).
\]


\subsection{A free map \(\bT^3\to\mathbb R^9\)}
We use coordinates \((x,y,z)\in \bT^3\). 
Set
\[
R_3(x,y)=
\operatorname{diag}
\bigl(
A(x),\,
A(y),\,
A(x+y),\,
A(x-y),\,
1
\bigr)
\in SO(9)
\]
and define $F_3:T^3\longrightarrow \mathbb R^9$ as
$$
F_3(x,y,z)=R_3(x,y)v_3(z)
$$
for some $v_3(z)\in\bR^9$.
For \(i=1,2\), define
\[
X_1=
\operatorname{diag}(J,0,J,J,0),
\qquad
X_2=
\operatorname{diag}(0,J,J,-J,0).
\]
Then
\[
\frac{\partial R_3}{\partial x}=R_3X_1,
\qquad
\frac{\partial R_3}{\partial y}=R_3X_2,
\]
so that
\[
\partial_xF_3=R_3X_1v_3,\qquad
\partial_yF_3=R_3X_2v_3,\qquad
\partial_zF_3=R_3v_3',
\]
\[
\partial_{xx}F_3=R_3X_1^2v_3,\qquad
\partial_{xy}F_3=R_3X_1X_2v_3,\qquad
\partial_{xz}F_3=R_3X_1v_3',
\]
\[
\partial_{yy}F_3=R_3X_2^2v_3,\qquad
\partial_{yz}F_3=R_3X_2v_3',
\qquad
\partial_{zz}F_3=R_3v_3''.
\]
Since \(R_3(x,y)\in SO(9)\), the determinant of the osculating matrix is
\[
\det \cD F_3(x,y,z)=D(z),
\]
where
\[
D(z)=
\det
\bigl(
X_1v_3,\,
X_2v_3,\,
v_3',\,
X_1^2v_3,\,
X_1X_2v_3,\,
X_1v_3',\,
X_2^2v_3,\,
X_2v_3',\,
v_3''
\bigr).
\]
Thus, the determinant depends only on \(z\).

A computer-aided search of a viable immersion $v_3$ led to the following expression:
\[
v_3(z)=
\begin{pmatrix}
1+\sin z\\
1\\
1\\
1-2\cos z\\
-\cos z\\
-1-\cos z\\
0\\
2\sin z\\
\cos z
\end{pmatrix}.
\]

A direct computation gives
$$
D(z) = -36+\sin z-4\cos(2z)-3\cos(3z)-\frac{15}{2}\sin(3z)+12\sin(4z)+3\cos(5z)-\frac{1}{2}\sin(5z).
$$
As in case of $F_2$, we get that
$$
D(z) \leq -36 + 1 + 4 + 3 + \frac{15}{2} + 12 + 3 + \frac{1}{2} < 0.
$$


Hence, $\D F_3$ is non-degenerate at every point, namely

\[
F_3\in \operatorname{Free}^{\infty}(T^3,\mathbb R^9).
\]

\subsection{A free map \(\T^4\to \mathbb R^{14}\)}
\label{sec: T4}
We use coordinates
$
(x_1,x_2,x_3,z)\in \T^4
$
and the six weights
\[
\begin{aligned}
w_1&=(1,0,0),&
w_2&=(0,1,0),&
w_3&=(0,0,1),\\
w_4&=(1,1,0),&
w_5&=(1,0,1),&
w_6&=(0,1,1).
\end{aligned}
\]
and define
\[
R_4(x_1,x_2,x_3)
=
\operatorname{diag}
\bigl(
A(\langle w_1,x\rangle),\ldots,
A(\langle w_6,x\rangle),I_2
\bigr)
\in SO(14).
\]
Equivalently,
\[
R_4(x_1,x_2,x_3)
=
\]
\[
diag\left(
A(x_1),
A(x_2),
A(x_3),
A(x_1+x_2),
A(x_1+x_3),
A(x_2+x_3),
I_2
\right)
\]

We set
\[
F_4:\T^4\longrightarrow\mathbb R^{14},
\qquad
F_4(x_1,x_2,x_3,z)
=
R_4(x_1,x_2,x_3)v_4(z)
\]
and we use the notation
\[
s=\sin z,\qquad c=\cos z,\qquad h=\cos(3z).
\]
Define
\[
\begin{aligned}
\lambda_1&=c+(s+c)h,\\
\lambda_2&=s,\\
\lambda_3&=c,\\
\lambda_4&=s+c+sh,\\
\lambda_6&=0
\end{aligned}
\]
and set
\[
B=
2\lambda_1\lambda_2\lambda_3
-\lambda_1\lambda_3\lambda_4
-\lambda_2\lambda_3\lambda_4
\]
and
\[
\lambda_5=-\frac18-B.
\]
Finally, let
\[
\bar\lambda=\frac16\sum_{j=1}^6\lambda_j,
\qquad
\mu_j=2(\lambda_j-\bar\lambda),
\]
and define
\[
f_j(z)=\int_0^z\mu_j(t)\,dt,
\qquad j=1,\ldots,6.
\]

The functions \(f_j\) are \(2\pi\)-periodic. Indeed, the functions
\(\lambda_1,\lambda_2,\lambda_3,\lambda_4,\lambda_6\) have zero average, and
\[
B=-c-c^2(c+s)\cos(3z)-cs(c+s)\cos^2(3z).
\]
By trigonometric orthogonality,
\[
\frac1{2\pi}\int_0^{2\pi}B(z)\,dz=-\frac18,
\]
so \(\lambda_5\) also has zero average. Hence each \(\mu_j\) has zero average.
Moreover,
\[
\sum_{j=1}^6\mu_j=0,
\]
and therefore, with the above choice of primitives,
\[
\sum_{j=1}^6 f_j(z)=0
\qquad\text{for every }z.
\]

Finally, we define
\[
v_4(z)=
\bigl(
e^{f_1(z)},0,\ldots,e^{f_6(z)},0,\cos z,\sin z
\bigr)^T
\in\R^{14}.
\]
A direct computation shows that
$$
\det\cD F_4=1
$$ 
identically on $\bT^4$, so that 
$$
F_4\in\Free^\infty(\bT^4,\bR^{14}).
$$

\medskip
In this case, there is an argument that allows to evaluate $D_4(z)$ without the need of calculating directly a $14\times14$ determinant.
We illustrate this argument below.

For \(i=1,2,3\), put
\[
X_i=
\operatorname{diag}
(w_{1i}J,\ldots,w_{6i}J,0_2).
\]
Since
\[
\partial_{x_i}R=RX_i,
\]
all fourteen columns of the osculating matrix have the common left factor
\(R(x)\). 
Since \(R(x)\in SO(14)\), this factor does not affect the determinant.

After removing \(R(x)\), the columns are
\[
X_1v,\ X_2v,\ X_3v,\ v',
\]
\[
X_iX_jv,\quad 1\le i\le j\le3,
\]
and
\[
X_1v',\ X_2v',\ X_3v',\ v''.
\]

The last fixed two-dimensional block is
\[
(\cos z,\sin z).
\]
Only the columns \(v'\) and \(v''\) have nonzero entries in the corresponding
two rows, and their \(2\times2\) minor is
\[
\det
\begin{pmatrix}
-\sin z&-\cos z\\
\cos z&-\sin z
\end{pmatrix}
=1.
\]
Expanding along these two rows therefore reduces the determinant to the first
six moving planes.
Let
\[
W=
\begin{pmatrix}
1&0&0\\
0&1&0\\
0&0&1\\
1&1&0\\
1&0&1\\
0&1&1
\end{pmatrix}.
\]
The six pure second derivatives produce the quadratic evaluation matrix
\[
Q=
\bigl(w_{ki}w_{kj}\bigr)_{
\substack{1\le k\le6\\1\le i\le j\le3}}
=
\begin{pmatrix}
1&0&0&0&0&0\\
0&0&0&1&0&0\\
0&0&0&0&0&1\\
1&1&0&1&0&0\\
1&0&1&0&0&1\\
0&0&0&1&1&1
\end{pmatrix},
\]
where the quadratic monomials are ordered as
\[
(11),(12),(13),(22),(23),(33).
\]
A direct calculation gives
\[
{\det Q=-1.}
\]

For the \(j\)-th moving plane,
\[
v_j=(e^{f_j},0),
\qquad
v_j'=(\mu_je^{f_j},0).
\]
Thus
\[
X_iv_j=e^{f_j}w_{ji}(0,1),
\]
\[
X_iX_kv_j=-e^{f_j}w_{ji}w_{jk}(1,0),
\]
and
\[
X_iv_j'=\mu_je^{f_j}w_{ji}(0,1).
\]
Factoring \(e^{f_j}\) from each radial and tangential row gives the factor
\[
\prod_{j=1}^6 e^{2f_j}
=
e^{2\sum_jf_j}
=1.
\]
With the indicated ordering of rows and columns, the remaining determinant is
\[
\det Q\cdot P(\mu_1,\ldots,\mu_6),
\]
where
\[
P(\xi_1,\ldots,\xi_6)
=
\det\!\left[
W\ |\ 
\operatorname{diag}(\xi_1,\ldots,\xi_6)W
\right].
\]

We now compute \(P(\mu)\). Since \(\lambda_6=0\), substitution into \(P\) gives
\[
\begin{aligned}
P(\lambda)
={}&
2\lambda_1\lambda_2\lambda_3
-\lambda_1\lambda_2\lambda_5
-\lambda_1\lambda_3\lambda_4\\
&-\lambda_2\lambda_3\lambda_4
-\lambda_2\lambda_3\lambda_5
+\lambda_2\lambda_4\lambda_5
+\lambda_3\lambda_4\lambda_5.
\end{aligned}
\]
Hence
\[
P(\lambda)=B+\mathcal A\,\lambda_5,
\]
where
\[
\mathcal A
=
-\lambda_1\lambda_2-\lambda_2\lambda_3
+\lambda_2\lambda_4+\lambda_3\lambda_4.
\]
But
\[
\begin{aligned}
\mathcal A
&=
\lambda_3(\lambda_4-\lambda_2)
+\lambda_2(\lambda_4-\lambda_1)\\
&=
c(c+sh)+s(s-ch)\\
&=
c^2+s^2\\
&=1.
\end{aligned}
\]
Therefore, by the definition of \(\lambda_5\), $P(\lambda)=-\frac18$.
There are two simple algebraic properties of \(P\). 
First,
\[
P(\xi+t(1,\ldots,1))=P(\xi)
\]
for every \(t\), because
\[
\operatorname{diag}(\xi+t\mathbf1)W
=
\operatorname{diag}(\xi)W+tW,
\]
and the term \(tW\) can be removed by column operations. Second, \(P\) is
homogeneous of degree \(3\). Since
\[
\mu=2(\lambda-\bar\lambda\mathbf1),
\]
we obtain
\[
P(\mu)
=
2^3P(\lambda-\bar\lambda\mathbf1)
=
8P(\lambda)
=
-1.
\]

Putting everything together,
\[
\det \cD F_4
=
(\det Q)
\left(\prod_{j=1}^6e^{2f_j}\right)
P(\mu)
\det
\begin{pmatrix}
-\sin z&-\cos z\\
\cos z&-\sin z
\end{pmatrix}.
\]
Thus
\[
\det \cD F_4=(-1)(1)(-1)(1)=1
\]
at every point.
\subsection{A free map \(\T^5\to \mathbb R^{20}\)}
\label{sec: T5}
We use coordinates
$
(x_1,x_2,x_3,x_4,z)\in \T^5.
$
We use the ten weights
\[
\begin{aligned}
&w_1=(1,0,0,0),\qquad
w_2=(0,1,0,0),\qquad
w_3=(0,0,1,0),\\
&w_4=(1,1,0,0),\qquad
w_5=(1,0,1,0),\qquad
w_6=(0,1,1,0),\\
&w_7=(1,0,1,1),\qquad
w_8=(0,0,0,1),\\
&w_9=(1,0,0,1),\qquad
w_{10}=(0,1,0,1).
\end{aligned}
\]
and define
\[
R_5(x_1,x_2,x_3,x_4)
=
\operatorname{diag}
\bigl(
A(\langle w_1,x\rangle),
A(\langle w_2,x\rangle),
\dots,
A(\langle w_{10},x\rangle)
\bigr)
\in SO(20),
\]
where \(x=(x_1,x_2,x_3,x_4)\). Equivalently,
\[
\begin{aligned}
R_5(x_1,x_2,x_3,x_4)
=
\operatorname{diag}\bigl(
&A(x_1),A(x_2),A(x_3),
A(x_1+x_2),A(x_1+x_3),A(x_2+x_3),\\
&A(x_1+x_3+x_4),A(x_4),A(x_1+x_4),A(x_2+x_4)
\bigr).
\end{aligned}
\]

Now set
\[
F_5:\T^5\longrightarrow\mathbb R^{20},
\qquad
F_5(x_1,x_2,x_3,x_4,z)
=
R_5(x_1,x_2,x_3,x_4)v_5(z).
\]

For \(i=1,2,3,4\), put
\[
X_i=
\operatorname{diag}(w_{1i}J,w_{2i}J,\ldots,w_{10,i}J).
\]
Then
\[
\frac{\partial R_5}{\partial x_i}=R_5X_i.
\]
Therefore the twenty derivative vectors of \(F_5\) are obtained by applying the common matrix \(R_5\) to the twenty vectors
\[
X_1v_5,\quad X_2v_5,\quad X_3v_5,\quad X_4v_5,\quad v_5',
\]
\[
X_iX_jv_5,\qquad 1\le i\le j\le 4,
\]
\[
X_1v_5',\quad X_2v_5',\quad X_3v_5',\quad X_4v_5',\quad v_5''.
\]
Hence
\[
\det \cD F_5(x_1,x_2,x_3,x_4,z)=D(z),
\]
where
\[
D(z)=
\det\Bigl(
X_1v_5,X_2v_5,X_3v_5,X_4v_5,v_5',
(X_iX_jv_5)_{1\le i\le j\le4},
X_1v_5',X_2v_5',X_3v_5',X_4v_5',v_5''
\Bigr).
\]

Let
\[
s=\sin z,\qquad c=\cos z.
\]

A computer-aided search of viable immersed loops $v$ led to the following expression.
Write \(v_5=(u_1,\ldots,u_{10})\), with \(u_j\in\mathbb R^2\), and take
$$ u_4(z)=(1,\sin z),\qquad u_7(z)=(1,-\cos z), $$
and
$$ u_j(z)=(e^{f_j(z)},0) $$
for the other eight indices, where
\begin{enumerate}
    \item[] $f_1(z) = -\sin z + \cos z$;
    \item[] $f_2(z) = -\sin z$;
    \item[] $f_3(z) = -\cos z$;
    \item[] $f_5(z) = 0$;
    \item[] $f_6(z) = \sin z$;
    \item[] $f_8(z) = 0$;
    \item[] $f_9(z) = \sin z + \cos z$;
    \item[] $f_{10}(z) = -\cos z$.
\end{enumerate}

A direct computation shows that
$$
\det\cD F_5=1
$$ 
identically on $\bT^5$, so that 
$$
F_5\in\Free^\infty(\bT^5,\bR^{20}).
$$

\medskip




Similarly to the case of $F_4$, there is an argument that allows to evaluate $D(z)$ without the need of directly calculating a $20\times20$ determinant.
We illustrate this argument below.

For a planar block $u\in\mathbb{R}^2$, write
\[
[u,v]=\det(u,v).
\]
Expanding first along the ten second-derivative columns
\[
X_iX_jv_5,\qquad 1\leq i\leq j\leq 4,
\]
gives the quadratic evaluation matrix
\[
Q=
\bigl(w_{ki}w_{kj}\bigr)_{
\substack{
1\leq k\leq 10\\
1\leq i\leq j\leq 4
}}.
\]
With the quadratic monomials ordered as
\[
(11),(12),(13),(14),(22),(23),(24),(33),(34),(44),
\]
a direct calculation gives
\[
\det Q=-1.
\]

After the corresponding block expansion, the remaining determinant is
represented row by row by
\[
\bigl(
|u_j|^2w_j,\,
[u_j,u_j'],\,
\langle u_j,u_j'\rangle w_j,\,
[u_j,u_j'']
\bigr).
\]

For the two exceptional blocks
\[
u_4(z)=(1,\sin z),
\qquad
u_7(z)=(1,-\cos z),
\]
we have
\[
[u_4,u_4']=\cos z,
\qquad
[u_4,u_4'']=-\sin z,
\]
and
\[
[u_7,u_7']=\sin z,
\qquad
[u_7,u_7'']=\cos z.
\]
Hence their contribution to the two scalar columns is
\[
\det
\begin{pmatrix}
\cos z&-\sin z\\
\sin z&\cos z
\end{pmatrix}
=1.
\]

Write
\[
f_j(z)=a_j\sin z-b_j\cos z.
\]
In the order
\[
j=1,2,3,5,6,8,9,10,
\]
the corresponding coefficient vectors are
\[
a=(-1,-1,0,0,1,0,1,0),
\]
and
\[
b=(-1,0,1,0,0,0,-1,1).
\]
Thus
\[
\lambda_j:=f_j'
=a_j\cos z+b_j\sin z.
\]

Moreover,
\[
\sum_j a_j=0,
\qquad
\sum_j b_j=0,
\]
and therefore
\[
\sum_j f_j(z)=0.
\]
Consequently,
\[
\prod_j e^{2f_j(z)}=1.
\]

Let $W$ be the $8\times 4$ matrix whose rows are the eight weights
corresponding to
\[
j=1,2,3,5,6,8,9,10.
\]
Thus
\[
W=
\begin{pmatrix}
1&0&0&0\\
0&1&0&0\\
0&0&1&0\\
1&0&1&0\\
0&1&1&0\\
0&0&0&1\\
1&0&0&1\\
0&1&0&1
\end{pmatrix}.
\]

For each exponential block,
\[
u_j=(e^{f_j},0),
\]
we have
\[
|u_j|^2=e^{2f_j},
\qquad
[u_j,u_j']=[u_j,u_j'']=0,
\qquad
\langle u_j,u_j'\rangle
=\lambda_j e^{2f_j}.
\]
After factoring $e^{2f_j}$ from each of these eight rows, the remaining
$8\times8$ determinant is
\[
\det\bigl[
W\mid\operatorname{diag}(\lambda_j)W
\bigr].
\]

More generally, for indeterminates $X,Y$, a direct calculation gives
\[
\det\bigl[
W\mid
\operatorname{diag}(a_jX+b_jY)W
\bigr]
=(X^2+Y^2)^2.
\]
Setting
\[
X=\cos z,
\qquad
Y=\sin z,
\]
we obtain
\[
\det\bigl[
W\mid\operatorname{diag}(\lambda_j)W
\bigr]
=1.
\]

Putting together the quadratic block, the two exceptional rows, and the
eight exponential rows, with the above row and column ordering, gives
\[
\det\mathcal{D}F_5
=
-(\det Q)
\left(
\prod_j e^{2f_j(z)}
\right)
\det\bigl[
W\mid\operatorname{diag}(\lambda_j)W
\bigr]
\det
\begin{pmatrix}
\cos z&-\sin z\\
\sin z&\cos z
\end{pmatrix}.
\]
Therefore
\[
\det\mathcal{D}F_5
=
-(-1)(1)(1)(1)=1
\]
at every point of $\bT^5$. 

\subsection{Obstructions and extensions}
\label{sec: obstructions}

{\BF There is an obstruction to this method for $m>5$.} 
The ansatz used in this section does not work on $\bT^m$ for $m>5$.
Indeed, let $k=m-1$. 
Then the ansatz has the form
\[
F(x_1,\ldots,x_k,z)=R(x_1,\ldots,x_k)v(z),
\]
where \(v:\bS^1\to\mathbb R^{q_m}\) is a loop and \(R\) is a block-diagonal
orthogonal matrix made of planar rotations:
\[
R(x)=
\operatorname{diag}
\bigl(
A(\langle w_1,x\rangle),
\ldots,
A(\langle w_r,x\rangle),
I_s
\bigr),\;2r+s=q_m.
\]
Here
\[
w_j=(w_{j1},\ldots,w_{jk})\in\mathbb Z^k,
\qquad
r\leq\left\lfloor \frac{q_m}{2}\right\rfloor .
\]
For \(i=1,\ldots,k\), set
\[
X_i=
\operatorname{diag}
\bigl(
w_{1i}J,\ldots,w_{ri}J,0_s
\bigr),
\qquad
J=
\begin{pmatrix}
0&-1\\
1&0
\end{pmatrix}.
\]
Then
\[
\frac{\partial R}{\partial x_i}=RX_i.
\]
Since \(R\) is invertible, the rank of the osculating matrix of \(F\) is the same
as the rank of the matrix whose columns are obtained after removing the common
left factor \(R\).

In particular, the pure second derivatives in the \(x\)-variables are
\[
\partial_{x_i x_j}F=R X_iX_j v(z),
\qquad
1\leq i\leq j\leq k.
\]
Thus, after removing the common factor \(R\), the relevant columns are
\[
X_iX_j v(z),
\qquad
1\leq i\leq j\leq k.
\]
Now look at one rotation block. If \(a_j(z)\in\mathbb R^2\) denotes the
\(j\)-th two-dimensional component of \(v(z)\), then
\[
\bigl(X_iX_\ell v(z)\bigr)_j
=
-w_{ji}w_{j\ell}\,a_j(z).
\]
Hence, for each block \(j\), the pure second derivatives contribute only one
scalar row, namely the quadratic vector
\[
\bigl(w_{j1}^2,\,
w_{j1}w_{j2},\,
\ldots,\,
w_{jk}^2\bigr)
\in \operatorname{Sym}^2(\mathbb R^k).
\]
Therefore the rank of the family
\[
\{X_iX_j v(z)\}_{1\leq i\leq j\leq k}
\]
is at most the number \(r\) of rotation blocks. On the other side, there are
\[
\dim \operatorname{Sym}^2(\mathbb R^k)
=
\frac{k(k+1)}2
=
\frac{m(m-1)}2
\]
pure second derivatives in the \(x\)-variables. 
For the osculating matrix to have full rank, these columns must in particular be linearly independent. Hence a necessary condition for this ansatz to work is
\[
r\geq \frac{m(m-1)}2.
\]
Since
\[
r\leq\left\lfloor \frac{q_m}{2}\right\rfloor
=
\left\lfloor \frac{m(m+3)}4\right\rfloor,
\]
we obtain the necessary condition
\[
\left\lfloor \frac{m(m+3)}4\right\rfloor
\geq
\frac{m(m-1)}2.
\]
For \(m>5\), however,
\[
\frac{m(m+3)}4
<
\frac{m(m-1)}2,
\]
because this inequality is equivalent to
\[
m+3<2m-2,
\]
that is, to \(m>5\). 
Therefore the necessary condition fails for every
\(m>5\).

Consequently, the simple block-rotation ansatz
\[
F(x,z)=R(x)v(z),
\]
with \(v\) depending on only one variable and \(R\) built only out of planar rotation blocks and possibly a fixed identity block, cannot produce free maps
\[
\bT^m\longrightarrow \mathbb \bR^{q_m}
\]
in critical dimension for \(m>5\).

\paragraph{\bf A modification of the ansatz that might work for $m>5$.}
A natural extension of the ansatz used so far is to replace $v:\bS^1\to\bR^{q_m}$ by some $u:\bT^s\to\bR^{q_m}$ with $s>1$.
As an example, below we use this extended method with $s=2$ to build a free map $\bT^4\to\bR^{14}$.

We use coordinates
$
(x_1,x_2,u,v)\in \bT^4
$
and
consider the five weights
\[
w_1=(-2,-1),\;
w_2=(2,0),\;
w_3=(1,-2),\;
w_4=(-1,-2),\;
w_5=(0,-2).
\]
Define
\[
R_{10}(x_1,x_2)
=
\operatorname{diag}
\bigl(
R_{-2x_1-x_2},
R_{2x_1},
R_{x_1-2x_2},
R_{-x_1-2x_2},
R_{-2x_2}
\bigr)
\in SO(10).
\]
Equivalently,
\[
R_{10}(x_1,x_2)
=
\operatorname{diag}
\bigl(
R_{\langle w_1,x\rangle},
R_{\langle w_2,x\rangle},
R_{\langle w_3,x\rangle},
R_{\langle w_4,x\rangle},
R_{\langle w_5,x\rangle}
\bigr),
\qquad x=(x_1,x_2).
\]

We now construct \(b:\bT^2\to\mathbb R^{10}\) and
\(c:\bT^2\to\mathbb R^4\).  Set
\[
c(u,v)=(\cos u,\sin u,\cos v,\sin v).
\]
Write
\[
b(u,v)=(b_1(u,v),\ldots,b_5(u,v)),
\qquad
b_r(u,v)\in\mathbb R^2,
\]
and define
\[
b_r(u,v)=\rho_r(u)
\begin{pmatrix}
\cos(\mu_r v)\\
\sin(\mu_r v)
\end{pmatrix},
\]
where
\[
(\mu_1,\mu_2,\mu_3,\mu_4,\mu_5)=(2,2,-2,-3,1).
\]
We take
\[
\rho_r(u)=e^{Q_r(u)}.
\]
Let
\[
q_r(u)=Q_r'(u).
\]
The logarithmic derivatives are chosen as
\[
\begin{aligned}
q_1&=-\cos u+\cos(2u),\\
q_2&=-\sin u+\cos(2u),\\
q_3&=\phantom{-}\sin u+\cos(2u),\\
q_4&=\phantom{-}\sin(2u),\\
q_5&=\phantom{-}\cos u+\cos(2u).
\end{aligned}
\]
Equivalently, one may choose the following periodic primitives:
\[
\begin{aligned}
Q_1&=-\sin u+\frac{1}{2}\sin(2u),\\
Q_2&=\phantom{-}\cos u+\frac{1}{2}\sin(2u),\\
Q_3&=-\cos u+\frac{1}{2}\sin(2u),\\
Q_4&=-\frac{1}{2}\cos(2u),\\
Q_5&=\phantom{-}\sin u+\frac{1}{2}\sin(2u).
\end{aligned}
\]
Thus, each \(\rho_r\) is smooth, positive and periodic.

Set now
\[
F(x_1,x_2,u,v)
=
\operatorname{diag}\bigl(R_{10}(x_1,x_2),I_4\bigr)
\begin{pmatrix}
b(u,v)\\
c(u,v)
\end{pmatrix}.
\]

For \(i=1,2\), set
\[
X_i=\operatorname{diag}(w_{1i}J,w_{2i}J,w_{3i}J,w_{4i}J,w_{5i}J).
\]
Then
\[
\frac{\partial R_{10}}{\partial x_i}=R_{10}X_i.
\]
The four columns supplied by the \(I_4\)-block are
\[
c_u,\qquad c_v,\qquad c_{uu},\qquad c_{vv}.
\]
They are linearly independent, since
\[
\det(c_u,c_v,c_{uu},c_{vv})=-1.
\]
Therefore, after expanding the full \(14\times14\) osculating determinant along
the \(I_4\)-block, it is enough to check the \(10\times10\) determinant
\[
\Delta_b
=
\det\bigl(
X_1b,\ X_2b,\ 
X_1^2b,\ X_1X_2b,\ X_2^2b,\ 
X_1b_u,\ X_1b_v,\ 
X_2b_u,\ X_2b_v,\ 
b_{uv}
\bigr).
\]

A direct block computation gives
\[
\Delta_b
=
4784\,e^{2(Q_1+\cdots+Q_5)}
P(q_1,\ldots,q_5),
\]
with
\[
\begin{aligned}
P(q_1,\ldots,q_5)={}&-95-\cos u-5\sin u+30\sin(2u)+3\cos(3u)+11\sin(3u)+\\
&+3\cos(4u)+23\sin(4u)-6\cos(5u)-4\sin(5u).
\end{aligned}
\]

Since 
$$
P(q_1,\dots,q_5)\leq-95+1+5+30+3+11+3+23+6+4<0,
$$
we have that
\[
\Delta_b<0
\qquad\text{for all }(u,v)\in \bT^2.
\]
Consequently the full osculating determinant of \(F\) is everywhere nonzero, and
so
\[
F\in \operatorname{Free}^{\infty}(T^4,\mathbb R^{14}).
\]

\medskip
We have not found examples of free maps on $\bT^6$ and $\bT^7$ using this extended ansatz.
If the reason is that the set of free maps for $m>5$ is ``very small'', a systematic numerical exploration might be able to find them.

\paragraph{\(k\)-free smooth maps on \(\bT^2\) in critical dimension.}
Given a manifold $M$ of dimension $m$, a smooth map \(F:M\to\mathbb R^q\) is {\BF\(k\)-free} if, in local coordinates, the partial derivatives of $F$ up to order $k$ are linearly independent at every point.
In particular, 2-free maps are simply free maps.
The critical dimension for the existence of $k$-free maps on $M$, namely the number of partial derivative of order from 1 up to $k$, is
\[
\text{\BF$q_{m,k}$}
=
\sum_{j=1}^{k} \binom{m+j-1}{j}
=
\binom{m+k}{k}-1.
\]

In~\cite{Cos90}, S.I. Rodriguez Costa showed that there exist $k$-free maps on $\bS^1$ in critical dimension for every $k$.
Here, we show that our method allows to find $k$-free maps on $\bT^2$ in critical dimension.
In the sequel, whenever \(k\)-osculating matrices are used, their columns are ordered by increasing total order of differentiation. 
We denote the $k$-osculating $q_{m,k}\times q$ matrix of a map $F:M\to\bR^q$ by $\cD^k F$ sorted in increasing order of derivative and, for each order, in lexicographic order.

\medskip
We discuss first the case $k=3$.
The critical dimension for a 3-free map on $\bT^2$ is 
$$
q_{2,3}=\binom{5}{2}-1 = 9.
$$

Define
\[
R_3(x)=\operatorname{diag}\bigl(A(x),A(2x),1,A(3x),A(4x)\bigr)\in SO(9).
\]
Then, if
\[
X=\operatorname{diag}(J,2J,0,3J,4J),
\]
we have that
\[
R'_3(x)=R_3(x)X.
\]

Set
\[
G_3:\bT^2\longrightarrow \mathbb R^9,
\qquad
G_3(x,y)=R_3(x)w_3(y).
\]
Then 
\[
\partial_xG_3=R_3Xw_3,\qquad \partial_yG_3=R_3w'_3,
\]
\[
\partial_{xx}G_3=R_3X^2w_3,\qquad \partial_{xy}G_3=R_3Xw'_3,\qquad \partial_{yy}G_3=R_3w''_3,
\]
\[
\partial_{xxx}G_3=R_3X^3w_3,\qquad
\partial_{xxy}G_3=R_3X^2w'_3,\qquad
\partial_{xyy}G_3=R_3Xw''_3,\qquad
\partial_{yyy}G_3=R_3w'''_3.
\]
Since \(R_3(x)\in SO(9)\), the common factor \(R_3(x)\) does not contribute to the osculating determinant:
\[
\det\cD^3 G_3 = \det R_3 
\det\bigl(
Xw_3,\ w'_3,\ X^2w_3,\ Xw'_3,\ w''_3,\ X^3w_3,\ X^2w'_3,\ Xw''_3,\ w'''_3
\bigr).
\]
Notice that the right hand side, as usual, does not depend on $x$ since $\det R_3=1$ and $X$ is constant; we denote the second factor by $D_3(y)$.


A computer-aided search led to many free maps in this form, the simplest of which is the one below:
\[
w_3(y)=
\begin{pmatrix}
2\\
1+\sin y\\
2\cos y-3\sin y\\
3\sin y\\
\cos y\\
1+\sin(2y)\\
\cos(2y)\\
1\\
0
\end{pmatrix}.
\]

A direct calculation shows that
$$
\begin{aligned}
    D_3(y) =&\; 1359360-248832\cos y-57600\sin y-115200\cos(2y)+\\
    &+69120\sin(2y)-41472\cos(3y)+16128\sin(3y).
\end{aligned}
$$
Hence,
$$
D_3(y)\geq 1359360-248832-57600-115200-69120-41472-16128>0,
$$
so that
$$
\det\cD^3 G_3>0,
$$
namely $G_3$ is a 3-free immersion of $\bT^2$ into $\bR^9$.



\medskip
We now consider the cases \(k=4,5,6\). We first give the details for
\(k=4\), since in this case the determinant is still short enough to be
displayed explicitly.

Set
\[
R_4(x)=\operatorname{diag}\bigl(A(x),A(2x),\ldots,A(7x)\bigr)
\in SO(14),
\]
and let
\[
X_4=\operatorname{diag}(J,2J,\ldots,7J).
\]
Define
\[
G_4:\bT^2\longrightarrow \mathbb R^{14},
\qquad
G_4(x,y)=R_4(x)w_4(y),
\]
where
\[
w_4(y)=
\begin{pmatrix}
\cos y\\
\sin y\\
\cos y\\
\sin y\\
\cos y\\
\sin y\\
\cos 2y\\
\sin 2y\\
\cos 2y\\
\sin 2y\\
71\cos y\\
-42\\
-34\cos y\\
58+42\cos y
\end{pmatrix}.
\]
Since \(R_4'(x)=R_4(x)X_4\), the determinant of the
\(4\)-osculating matrix is
\[
\det \cD^4G_4(x,y)=D_4(y),
\]
where
\[
D_4(y)
=
\det\bigl(X_4^p w_4^{(q)}(y)\bigr)_{
\substack{p,q\ge 0\\ 1\le p+q\le 4}},
\]
with the columns ordered as above.

Put
\[
t=\tan\frac y2.
\]
A direct symbolic computation gives
\[
D_4(y(t))
=
\frac{3072577536\,P_4(t)}{(1+t^2)^2},
\]
where
\[
\begin{aligned}
P_4(t)
={}&28638339847t^4
-59948663850t^3
+25921162800t^2  \\
&+148137235230t
+66788156513 .
\end{aligned}
\]
The polynomial \(P_4\) is strictly positive on \(\mathbb R\). Indeed,
\[
\begin{aligned}
P_4(t)
={}&
\frac{
\bigl(
28638339847t^2
-29974331925t
-28638339847
\bigr)^2
}{28638339847}
\\
&+
\frac{1484187513530275852793}{28638339847}
\left(
t+
\frac{1262787138900928889430}
{1484187513530275852793}
\right)^2
\\
&+
\frac{
939784670791147167558912791438
}{
1484187513530275852793
}.
\end{aligned}
\]
All three summands are nonnegative, and the last one is strictly positive.
Therefore
\[
P_4(t)>0
\qquad
\text{ for all }t\in\bR.
\]
The point \(y=\pi\), corresponding to \(t=\infty\), is also covered by the
positive leading coefficient of \(P_4\). Hence
\[
D_4(y)>0
\qquad
\text{ for all }y\in \bS^1.
\]
Therefore, \(G_4\) is a \(4\)-free map
in critical dimension.

\medskip
For \(k=5\) and \(k=6\), we use the same construction with
\[
R_5(x)=\operatorname{diag}\bigl(A(x),\ldots,A(10x)\bigr)\in SO(20),
\]
and
\[
R_6(x)=\operatorname{diag}\bigl(A(x),\ldots,A(13x),1\bigr)\in SO(27).
\]
Let \(w_6:\bS^1\to\mathbb R^{27}\) be the loop
\[
w_6(y)=
\begin{pmatrix}
w_4(y)\\
\cos y\\
\sin y\\
-\sin 3y\\
\cos 2y-\cos 3y\\
5\sin y-2\cos 2y\\
-7+2\sin 2y\\
\cos y\\
\sin y\\
\cos y\\
\sin y\\
\alpha(y)\\
\beta(y)\\
\gamma(y)
\end{pmatrix},
\]
where
\[
\begin{aligned}
\alpha(y)
={}&8035\cos y+341\sin 2y-108\cos 2y-317\cos 3y  \\
&-17\sin 4y-16\cos 4y-15\sin 5y+5\cos 5y,
\end{aligned}
\]
\[
\begin{aligned}
\beta(y)
={}&2850+4887\cos y+33\sin 2y-1244\cos 2y+314\sin 3y \\
&-11\cos 3y-11\sin 4y-16\cos 4y-4\sin 5y-15\cos 5y,
\end{aligned}
\]
and
\[
\gamma(y)
=
-349\cos 3y+85\sin 4y+177\cos 4y+30\sin 5y-5\cos 5y.
\]
For the \(5\)-free case, we take \(w_5:\bS^1\to\mathbb R^{20}\) to be the
projection of \(w_6\) onto its first \(20\) components.

Let
\[
X_5=\operatorname{diag}(J,2J,\ldots,10J)
\]
and
\[
X_6=\operatorname{diag}(J,2J,\ldots,13J,0).
\]
The corresponding determinants are
\[
D_5(y)
=
\det\bigl(X_5^p w_5^{(q)}(y)\bigr)_{
\substack{p,q\ge 0\\ 1\le p+q\le 5}},
\]
and
\[
D_6(y)
=
\det\bigl(X_6^p w_6^{(q)}(y)\bigr)_{
\substack{p,q\ge 0\\ 1\le p+q\le 6}},
\]
again with the columns ordered as above.

After the substitution
\[
t=\tan\frac y2,
\]
both \(D_5(y(t))\) and \(D_6(y(t))\) become rational functions of \(t\).
After clearing the positive powers of \(1+t^2\), one obtains polynomials
with rational coefficients. Exact symbolic computations, included in the
ancillary SageMath and Mathematica files, show via Sturm's algorithm that
these polynomials have no real roots and have constant nonzero sign.
Therefore the maps
\[
G_5(x,y)=R_5(x)w_5(y):\bT^2\longrightarrow \mathbb R^{20}
\]
and
\[
G_6(x,y)=R_6(x)w_6(y):\bT^2\longrightarrow \mathbb R^{27}
\]
are respectively \(5\)-free and \(6\)-free in critical dimension.

\medskip

Based on the results in~\cite{Cos90} and in this section, we pose the following conjecture:

\begin{conjecture}
    There exist smooth $k$-free maps $\bT^2\to\bR^{q_{2,k}}$ for every $k\geq2$.
\end{conjecture}

\section{Free maps in critical dimension on compact closed surfaces}
\label{sec: surfaces}
The main result of this section is that, given two compact closed smooth surfaces $M$ and $N$, the existence of smooth free maps in critical dimension on $M$ and $N$ implies the existence of free maps in critical dimension on their connected sum $M\# N$.
Due to the classification theorem for 2-manifolds, this will, in turn, allow us to show that free maps arise on every compact closed surface.



\subsection{The ``connected sum'' of two free maps}
It is a standard result that every closed compact orientable (resp. non-orientable) surface is the connected sum of some finite number of copies of $\bT^2$ (resp. $\bR P^2$).
Hence, in order to prove the existence of smooth free maps in any closed compact surface, it is enough to show that one can ``glue together'' nicely enough any two free maps on any given pair of surfaces.



\medskip
We shall use the following linear variant of the standard quadratic free map:
\[
\widetilde P:\mathbb R^2\to\mathbb R^5,
\qquad
\text{\BF$\widetilde P(x,y)$}
=
\left(
x,\ y,\ \frac{x^2-y^2}{2},\ xy,\ \frac{x^2+y^2}{2}
\right).
\]
This differs from
\[
P(x,y)=\left(x,y,\frac{x^2}{2},xy,\frac{y^2}{2}\right)
\]
by an invertible linear change in the last three target coordinates. Hence
\(\widetilde P\) is free.
In polar coordinates,
\[
x=r\cos\theta,\qquad y=r\sin\theta,
\]
we write
\[
\text{\BF$\widetilde W(r,\theta)$}
=
\widetilde P(r\cos\theta,r\sin\theta).
\]
Explicitly,
\[
\widetilde W(r,\theta)
=
\left(
r\cos\theta,\,
r\sin\theta,\,
\frac{r^2}{2}\cos2\theta,\,
\frac{r^2}{2}\sin2\theta,\,
\frac{r^2}{2}
\right).
\]

We also introduce the signed-cylinder model
\[
\text{\BF$C(u,\theta)$}
=
\left(
\cos\theta,\,
\sin\theta,\,
u\cos2\theta,\,
u\sin2\theta,\,
-\frac{u^2}{2}
\right).
\]
A direct computation gives
\[
\det\D C
=
2.
\]
Thus \(C\) is free on a full signed cylinder
\[
(-\epsilon,\epsilon)\times S^1.
\]

For \(\lambda>0\), define the anisotropic target dilation
\[
D_\lambda(X_1,X_2,X_3,X_4,X_5)
=
(\lambda X_1,\lambda X_2,\lambda^2X_3,\lambda^2X_4,\lambda^2X_5).
\]
We define the scaled signed-cylinder collar as
\[
\text{\BF$C_\lambda(u,\theta)$}=D_\lambda C(u,\theta),
\]
that is,
\[
C_\lambda(u,\theta)
=
\left(
\lambda\cos\theta,\,
\lambda\sin\theta,\,
\lambda^2u\cos2\theta,\,
\lambda^2u\sin2\theta,\,
-\frac{\lambda^2u^2}{2}
\right).
\]

The idea of the proof is the following.
In Lemma~\ref{lemma: 1} we show that there is a free map on the cylinder that is equal to $C(u,\theta)$ at one end of the cylinder and to $W(2u,\theta)$ at the other end.
Then, in Lemma~\ref{lem:gluing}, we show that, given any free map $f$ on a surface $M$, we can modify $f$ so that, near some disc, after replacing the disc with a cylinder collar, the modified free map coincides with $C_\lambda(u,\theta)$ for some $\lambda>0$.
At this point, in Theorem~\ref{thm: connected sum} we show how to glue two free maps on two surfaces $M$ and $N$ so to obtain a free map on $M\#N$.

\begin{lemma}[Interpolation from the polar collar to the signed cylinder]
\label{lemma: 1}
Let $u\in[0,1]$ and $\theta\in\bS^1$.
There exists a smooth free map
\[
A:[0,1]\times S^1\to\mathbb R^5
\]
such that, near \(u=0\),
\[
A(u,\theta)=C(u,\theta),
\]
and, near \(u=1\),
\[
A(u,\theta)=\widetilde W(2u,\theta).
\]
Consequently, for every \(\lambda>0\), the map
\[
A_\lambda:=D_\lambda\circ A
\]
is a free interpolation between \(C_\lambda\) near \(u=0\) and
\[
\widetilde W(2\lambda u,\theta)
\]
near \(u=1\).
\end{lemma}

\begin{proof}
Consider maps of the form
\[
F(u,\theta)
=
\left(
a(u)\cos\theta,\,
a(u)\sin\theta,\,
b(u)\cos2\theta,\,
b(u)\sin2\theta,\,
c(u)
\right).
\]
Set
\[
b(u)=u\,a(u).
\]
Then a direct calculation gives
\[
\det\D F
=
-2a(u)^3H(u),
\]
where
\[
H(u)
=
\bigl(a-3ua'\bigr)c''
+
\bigl(3ua''-2a'\bigr)c'.
\]

Choose
\[
a(u)=
\frac{
508u^9+2909u^8+2292u^7+986u^6
-53420u^5+76121u^4-28396u^3+1000
}{1000},
\]
and
\[
c(u)=
\frac{
u^2\bigl(
7214u^7+3187u^6+1311u^5-291u^4
-144573u^3+223056u^2-87404u-500
\bigr)
}{1000}.
\]
Then
\begin{align*}
&a(0)=1,\;\; \;\;a'(0)=0,\;\; \;\;a''(0)=\phantom{-}0,\\
&b(0)=0,\;\;\;\; b'(0)=1,\;\; \;\;b''(0)=\phantom{-}0,\\
&c(0)=0,\;\;\;\; c'(0)=0,\;\;\;\; c''(0)=-1,
\end{align*}
Notice that
\[
C(u,\theta)
=
\left(
\alpha(u)\cos\theta,\,
\alpha(u)\sin\theta,\,
\beta(u)\cos2\theta,\,
\beta(u)\sin2\theta,\,
\gamma(u)
\right)
\]
for
$$
\alpha(u)=1,\;\;\beta(u)=u,\;\;\gamma(u)=-u^2/2.
$$
Since the 2-jets of $\alpha,\beta,\gamma$ at $u=0$ coincide with the 2-jets of, respectively, $a,b,c$ at the same point, it follows that $j^2F=j^2C$ along $\{0\}\times\bS^1$, where the 2-jets are taken with respect to the variables $(u,\theta)$.

At the other endpoint,
\[
a(1)=2,\qquad a'(1)=2,\qquad a''(1)=0,
\]
\[
b(1)=2,\qquad b'(1)=4,\qquad b''(1)=4,
\]
\[
c(1)=2,\qquad c'(1)=4,\qquad c''(1)=4.
\]
Now, notice that 
\[
\widetilde W(2u,\theta)
=
\left(
\hat\alpha(u)\cos\theta,\,
\hat\alpha(u)\sin\theta,\,
\hat\beta(u)\cos2\theta,\,
\hat\beta(u)\sin2\theta,\,
\hat\gamma(u)
\right).
\]
for
$$
\hat\alpha(u) = 2u,\;\;
\hat\beta(u) = 2u^2,\;\;
\hat\gamma(u) = 2u^2.
$$
Since the 2-jets of $\hat\alpha,\hat\beta,\hat\gamma$ at $u=1$ coincide with the 2-jets of, respectively, $a,b,c$ at the same point, it follows that $j^2F=j^2W$ along $\{1\}\times\bS^1$, where the 2-jets are taken with respect to the variables $(u,\theta)$.

For the chosen $a(u)$ and $c(u)$, one can verify that
\[
H(u)=-\frac{K(u)}{500000},
\]
where
\[
\begin{aligned}
K(u)={}&
164912040u^{16}+2918300000u^{15}+3807304560u^{14}
+2517895353u^{13}\\
&-89566180787u^{12}+151601083986u^{11}
-80205126406u^{10}-10857198663u^9\\
&+119308348533u^8-243323486472u^7
+220917024720u^6\\
&-93123128268u^5+15542937404u^4
+1388938000u^3\\
&-1338336000u^2+262212000u+500000.
\end{aligned}
\]
By Sturm's theorem,
\[
a(u)>0
\text{ and }
K(u)>0
\]
for all $u\in[0,1]$.
Hence,
\[
\det\D F
=
\frac{a(u)^3K(u)}{250000}>0
\]
on \([0,1]\times S^1\).

At this stage, \(F\) agrees with the each endpoint model to order \(2\) at one of the two boundaries of the collar.
It remains to make the agreement with the endpoint models exact in
neighborhoods of the two boundary circles. 
Let
\[
h(u)=(a(u),b(u),c(u)).
\]
Let
\[
h_0(u)=\left(1,u,-\frac{u^2}{2}\right),
\]
and 
\[
h_1(u)=\left(2u,2u^2,2u^2\right).
\]
By the endpoint conditions already verified, we have that
\[
h(u)-h_0(u)=O(u^3)\qquad \text{as }u\to 0,
\]
together with the corresponding estimates for the first and second derivatives,
and similarly
\[
h(u)-h_1(u)=O((1-u)^3)\qquad \text{as }u\to 1,
\]
again with the corresponding first and second derivative estimates.

Now, choose a smooth function \(\rho:\mathbb R\to[0,1]\) such that
\[
\rho(s)=0 \quad \text{for } s\le 1,
\qquad
\rho(s)=1 \quad \text{for } s\ge 2.
\]
For \(0<\delta<1/4\), set
\[
\rho_0^\delta(u)=\rho\left(\frac{u}{\delta}\right),
\qquad
\rho_1^\delta(u)=\rho\left(\frac{1-u}{\delta}\right).
\]
First define
\[
h^{(0)}_\delta(u)
=
h_0(u)+\rho_0^\delta(u)\bigl(h(u)-h_0(u)\bigr).
\]
Thus \(h^{(0)}_\delta=h_0\) for \(0\le u\le \delta\), and
\(h^{(0)}_\delta=h\) for \(u\ge 2\delta\). Then define
\[
\widetilde h_\delta(u)
=
h_1(u)+\rho_1^\delta(u)
\bigl(h^{(0)}_\delta(u)-h_1(u)\bigr).
\]
Since the supports of the two transition regions are disjoint, this gives
\[
\widetilde h_\delta=h_0
\quad\text{for }0\le u\le \delta,
\]
\[
\widetilde h_\delta=h
\quad\text{for }2\delta\le u\le 1-2\delta,
\]
and
\[
\widetilde h_\delta=h_1
\quad\text{for }1-\delta\le u\le 1.
\]

Moreover,
\[
\|\widetilde h_\delta-h\|_{C^2([0,1])}\longrightarrow 0
\qquad\text{as }\delta\to 0.
\]
Indeed, on the transition region near \(u=0\), the differences
\(h-h_0\), \((h-h_0)'\), and \((h-h_0)''\) are respectively
\(O(\delta^3)\), \(O(\delta^2)\), and \(O(\delta)\), while derivatives of
\(\rho_0^\delta\) contribute at worst factors of \(\delta^{-1}\) and
\(\delta^{-2}\). Hence the resulting \(C^2\)-error is \(O(\delta)\).
The same argument applies near \(u=1\).

Replacing \(h=(a,b,c)\) by \(\widetilde h_\delta
=(\widetilde a_\delta,\widetilde b_\delta,\widetilde c_\delta)\), and setting
\[
\widetilde F_\delta(u,\theta)
=
\bigl(
\widetilde a_\delta(u)\cos\theta,\,
\widetilde a_\delta(u)\sin\theta,\,
\widetilde b_\delta(u)\cos 2\theta,\,
\widetilde b_\delta(u)\sin 2\theta,\,
\widetilde c_\delta(u)
\bigr),
\]
we obtain maps \(\widetilde F_\delta\) converging to \(F\) in the
\(C^2\)-topology as \(\delta\to0\). Since $\det\D F$
has a positive minimum on the compact cylinder \([0,1]\times S^1\),
freeness is preserved for all sufficiently small \(\delta\). 
Choose such a $\delta$ and set  $A=F_\delta$.
Then $A$ is free, agrees exactly with the signed-cylinder model \(C\) near \(u=0\), and agrees exactly with the polar model \(W(2u,\theta)\) near \(u=1\).
\end{proof}

\begin{lemma}[Normalization to the signed-cylinder collar]
\label{lem:gluing}
Let \(M\) be a smooth surface, let \(p\in M\), and let
\[
f\in\Free^\infty(M,\mathbb R^5).
\]
Then, for every sufficiently small \(\lambda>0\), after possibly postcomposing
\(f\) with an affine automorphism of \(\mathbb R^5\), there exists a disk
\(D\subset M\) centered at \(p\) and a free map
\[
\widehat f:M\setminus\operatorname{int}D\to\mathbb R^5
\]
such that:

\begin{enumerate}
\item \(\widehat f\) agrees with \(f\), up to the chosen affine automorphism,
outside an arbitrarily small neighborhood of \(D\);

\item near \(\partial D\), in collar coordinates \((u,\theta)\in[0,\epsilon)\times S^1\),
\[
\widehat f(u,\theta)=C_\lambda(u,\theta).
\]
\end{enumerate}
\end{lemma}

\begin{proof}
Choose local coordinates \((x,y)\) centered at \(p\). Since \(f\) is free, the
matrix
\[
\D f(0)
\]
is invertible. Hence, after translating the target and applying an invertible
linear map of \(\mathbb R^5\), we may assume that
\[
j^2_0f=j^2_0\widetilde P.
\]
Thus
\[
f(x,y)=\widetilde P(x,y)+R(x,y),
\]
where
\[
\partial_\alpha R(x,y)=O\bigl(|(x,y)|^{3-|\alpha|}\bigr),
\qquad
|\alpha|\le 2.
\]

Fix \(\lambda>0\) sufficiently small. On a small annulus around the circle
\[
r=2\lambda
\]
we interpolate between \(f\) and \(\widetilde P\) by
\[
f_\lambda(x,y)
=
\widetilde P(x,y)
+
\chi(r/\lambda)\bigl(f(x,y)-\widetilde P(x,y)\bigr),
\]
where \(r=\sqrt{x^2+y^2}\), and \(\chi\) is a smooth cutoff which is \(0\) near
\(r=2\lambda\) and \(1\) farther out.

Since \(f-\widetilde P=O(r^3)\) and the derivatives of the cutoff of
order \(j\) are \(O(\lambda^{-j})\), one has
\[
\|j^2f_\lambda-j^2\widetilde P\|_{C^0}=O(\lambda)
\]
on the interpolation annulus. For \(\lambda\) sufficiently small, \(f_\lambda\)
is therefore free, since \(\widetilde P\) is free and freeness is an open
condition on \(2\)-jets.

Thus, after this first modification, the map equals the polar standard model
\[
\widetilde W(r,\theta)
=
\widetilde P(r\cos\theta,r\sin\theta)
\]
near the circle \(r=2\lambda\).

Now attach the annular interpolation \(A_\lambda\) from the previous lemma.
Near its outer end, it equals
\[
\widetilde W(2\lambda u,\theta),
\]
so near \(u=1\) it agrees with the polar standard model near \(r=2\lambda\).
Near its inner end, it equals the signed-cylinder model
\[
C_\lambda(u,\theta).
\]

Replacing the corresponding punctured neighborhood by this annular collar gives
a free map on \(M\setminus\operatorname{int}D\), equal to the original map outside
a small neighborhood of \(p\), and equal to \(C_\lambda\) near \(\partial D\).
\end{proof}

\begin{theorem}[Connected-sum stability in minimal dimension for surfaces]
\label{thm: connected sum}
Let \(M\) and \(N\) be smooth compact connected surfaces. If
\[
\Free^\infty(M,\mathbb R^5)\neq\varnothing
\qquad\text{and}\qquad
\Free^\infty(N,\mathbb R^5)\neq\varnothing,
\]
then
\[
\Free^\infty(M\#N,\mathbb R^5)\neq\varnothing.
\]
\end{theorem}

\begin{proof}
Choose
\[
f\in\Free^\infty(M,\mathbb R^5),
\qquad
g\in\Free^\infty(N,\mathbb R^5).
\]
Choose points \(p\in M\), \(q\in N\). Applying the signed-cylinder normalization
lemma to \(f\) and \(g\), with the same value of \(\lambda>0\), we obtain disks
\[
D_M\subset M,
\qquad
D_N\subset N,
\]
and free maps
\[
\widehat f:M\setminus\operatorname{int}D_M\to\mathbb R^5,
\qquad
\widehat g:N\setminus\operatorname{int}D_N\to\mathbb R^5,
\]
such that near the two boundary circles
\[
\widehat f(u,\theta)=C_\lambda(u,\theta),
\]
and
\[
\widehat g(u,\theta)=C_\lambda(u,\theta),
\]
where \(u\ge 0\) is the inward collar coordinate on each punctured surface.

Let
\[
S=\operatorname{diag}(1,1,-1,-1,1).
\]
Then
\[
\det S=1,
\]
so postcomposition by \(S\) preserves freeness. Moreover
\[
S C_\lambda(u,\theta)=C_\lambda(-u,\theta).
\]
Replace \(\widehat g\) by \(S\widehat g\). Then near \(\partial D_N\),
\[
S\widehat g(u,\theta)=C_\lambda(-u,\theta).
\]

Now form the connected sum \(M\#N\) by identifying the two boundary circles
using the chosen angular parameter \(\theta\). 
Choosing the boundary angular parameters appropriately realizes the desired
connected-sum identification.

Define a map
\[
h:M\#N\to\mathbb R^5
\]
by
\[
h=
\begin{cases}
\widehat f & \text{on }M\setminus\operatorname{int}D_M,\\
S\widehat g & \text{on }N\setminus\operatorname{int}D_N.
\end{cases}
\]

Near the seam, introduce a signed collar coordinate \(\tau\) by
\[
\tau=u
\quad\text{on the \(M\)-side},
\qquad
\tau=-u
\quad\text{on the \(N\)-side}.
\]
Then on both sides of the seam,
\[
h(\tau,\theta)=C_\lambda(\tau,\theta).
\]
Therefore \(h\) is smooth across the seam. It is free near the seam because
\(C_\lambda\) is free, and it is free away from the seam because it agrees locally
with either \(\widehat f\) or \(S\widehat g\), both of which are free. Hence,
\[
h\in\Free^\infty(M\#N,\mathbb R^5).
\]
\end{proof}

We are now able to prove our main result.

\begin{theorem}
    \label{thm: main}
    Let $M$ be a compact closed surface. Then $\mathrm{Free}^\infty(M,\mathbb{R}^5) \neq \emptyset$.
\end{theorem}

\begin{proof}
Assume first that \(M\) is connected. By the classification theorem for closed
surfaces, \(M\) is diffeomorphic to exactly one of the following:
\[
\bS^2,\qquad \#_g T^2\quad (g\ge 1),\qquad \#_k \RP^2\quad (k\ge 1).
\]

We have:
\begin{itemize}
\item \(\Free^\infty(\bS^2,\mathbb R^5)\ne\varnothing\), by the Veronese map;
\item \(\Free^\infty(T^2,\mathbb R^5)\ne\varnothing\), by the explicit map above;
\item \(\Free^\infty(\RP^2,\mathbb R^5)\ne\varnothing\), by the Veronese map.
\end{itemize}

Iterated application of Theorem~\ref{thm: connected sum} gives the result for connected
\(M\). 
If \(M\) is disconnected, apply the connected case to each connected component.
\end{proof}

\begin{corollary}
    \label{cor: 1}
    Let $M$ be a compact closed surface.
    Then there is a free smooth embedding of $M$ into $\bR^5$.
\end{corollary}
\begin{proof}
By Theorem~\ref{thm: main}, the open set \(\Free^\infty(M,\mathbb   R^5)\) is nonempty.
Since freeness is a \(C^2\)-open condition and \(M\) is compact, \(\Free^\infty(M,\mathbb R^5)\) is open in the \(C^\infty\)-topology.
By Whitney approximation and transversality, since
$
    5=2\cdot 2+1,
$
embeddings \(M\to\mathbb R^5\) are \(C^\infty\)-dense in \(C^\infty(M,\mathbb R^5)\). 
Hence, any free map \(M\to\mathbb R^5\) admits a sufficiently small smooth perturbation which is still free and is an embedding.
\end{proof}

\paragraph{Free smooth embeddings.}
%
All maps on tori exhibited in Section~\ref{sec: tori} are smooth embeddings.

Consider for instance the map $F_2(x,y)=R_2(x) v_2(y):
\bT^2\to\bR^5$ 
and set coordinates $(\zeta^1,\dots,\zeta^5)$ on $\bR^5$.
The map $v_2:\bS^1\to\bR^5$ is an embedding because its projection over the $(\zeta^1,\zeta^4)$ plane is an embedding as well.
The toric action determined by $R_2$ acts as a rotation about the origin on the planes $(\zeta^1,\zeta^2)$ and $(\zeta^3,\zeta^4)$ and fixes the $\zeta^5$ axis.
Let $(x',y')$ be a point such that $F_2(x,y)=F_2(x',y')$.
Then, since $|\zeta^1(F_2)|^2+|\zeta^2(F_2)|^2=(1+\sin y)^2+1$ is invariant, the equation $(1+\sin y)^2=(1+\sin y')^2$ implies $\sin y=\sin y'$; since also $\zeta_5(F_2)$ is invariant, then also $\cos y=\cos y'$ and so $y=y'$.
Hence, we have that $R_2(x)(v_2(y))=R_2(x')(v_2(y))$, so also $x=x'$ and $F_2$ is globally injective.

One can prove in the same way that each of the maps $G_k$ and $F_k$ exhibited in Section~\ref{sec: T2} are smooth embeddings as well.


\section{Invariants of free maps in critical dimension}
\label{sec: invariants}

Given a parallelizable $M$, the osculating matrix of a free map $f:M\to\bR^{q_m}$ can be seen as a map $\cD f:M\to GL_{q_m}(\bR)$ after a choice of a parallelization.
On $\bT^m$, we fix the standard parallelization represented by the coordinate vector fields.
The map $f\mapsto Df$ is continuous from $\mathrm{Free}^\infty(\mathbb T^m,\mathbb R^{q_m})$
to $C^0(\mathbb T^m,GL_{q_m})$. In particular a path of free maps induces a homotopy of
osculating matrices, so $[Df]$ is constant along paths in $\mathrm{Free}^\infty$.
It is in fact locally constant. Indeed, let $f_0$ be free and set
$c=\big(\sup_{x}\|Df_0(x)^{-1}\|\big)^{-1}>0$, which is positive by compactness of
$\mathbb T^m$. If $f_1$ is free with $\sup_x\|Df_1(x)-Df_0(x)\|<c$, then
\[
(1-t)Df_0+tDf_1=Df_0\big[\,I+t\,Df_0^{-1}(Df_1-Df_0)\,\big]
\]
is invertible for every $t\in[0,1]$, since the bracket differs from the identity by an
operator of norm $<1$; the segment is therefore a homotopy in $GL_{q_m}(\mathbb R)$ and
$[Df_0]=[Df_1]$. By the continuity above, this inequality holds for every $f_1$ in a
suitable $C^2$-neighbourhood of $f_0$. Hence
\[
[Df]\in[\mathbb T^m,GL_{q_m}(\mathbb R)]
\]
is locally constant on $\mathrm{Free}^\infty(\mathbb T^m,\mathbb R^{q_m})$; in particular
free maps with different invariants lie in different connected components.


Below we evaluate $[\bT^m,GL_{q_m}(\bR)]$ for $m=2,3,4,5$ and then we investigate which values of the invariants can be achieved via the free maps exhibited in the present article.
In this section we will use the following well-known facts: 
\begin{enumerate}
    \item $\pi_0(O_n)=\mathbb{Z}_2$ and
      $\pi_i(O_n)=\pi_i(SO_n)$ for every $i\geq1$.
\item For $n\geq3$ the group $SO_n$ has a universal double cover $p:\mathrm{Spin}_n\to SO_n$ and $\pi_1(SO_n)=\mathbb{Z}_2$; 
\item $\pi_1(\mathrm{Spin}_n)=\pi_2(\mathrm{Spin}_n)=0$ for $n\geq3$; 
\item $\pi_3(\mathrm{Spin}_n)\simeq\mathbb{Z}$ for $n=5,9,14,20$.
\item $\pi_4(\mathrm{Spin}_{14})=\pi_4(\mathrm{Spin}_{20})=\pi_5(\mathrm{Spin}_{20})=0$.
\item $H^0(T^m;\mathbb{Z}_2)\simeq\mathbb{Z}_2$,\quad
      $H^1(T^m;\mathbb{Z}_2)\simeq\mathbb{Z}_2^{\,m}$,\quad
      $H^3(T^m;\mathbb{Z})\simeq\mathbb{Z}^{\binom{m}{3}}$.
    
\end{enumerate}
 
\begin{proposition}\label{thm:count}
For $2\le m\le 5$ there is a (non-canonical) bijection
\[
  [\bT^m,\GL_{q_m}(\R)]\ \leftrightarrow\
  H^0(\bT^m;\pi_0(O_{q_m}))
  \times
  H^1(\bT^m;\pi_1(O_{q_m}))
  \times
  H^3(\bT^m;\pi_3(O_{q_m})).
\]
\end{proposition}
 
\begin{proof}
{\bf The first invariant.}
The Gram--Schmidt algorithm is a deformation retraction of $\GL_{q_m}(\R)$ to $O_{q_m}$, so $[\bT^m,\GL_{q_m}(\R)]=[\bT^m,O_{q_m}]$.
Since $\bT^m$ is connected, the image of any continuous map lies in a single component, which is recorded by $\varepsilon=\sgn\det$; this is the factor $H^0(\bT^m,\pi_0(O_{q_m}))\simeq\Z_2$. 
Any fixed reflection identifies the non-identity component with $\SO_{q_m}$, so we reduced the problem to count $[\bT^m,\SO_{q_m}]$.
 
{\bf The second invariant.} 
After the deformation retract of $\GL_{q_m}(\R)$ to $O_{q_m}$, we compose it with $\cD f$ and possibly with a reflection in order to have its image in $SO_{q_m}$; to simplify notation, though, we still denote the map by $\cD f$.
In other words, from now on we assume that $\cD f:\bT^m\to SO_{q_m}$.

Now, since $\pi_1(SO_{q_m})\cong \mathbb Z_2$, the map $\cD f$ determines a homomorphism
\[
(\cD f)_*:\pi_1(\bT^m)=H_1(\bT^m;\mathbb Z)=\mathbb Z^m\longrightarrow \mathbb Z_2.
\]
This is an element of $H^1(\bT^m;\mathbb Z_2)$ which we denote by $\eta(f)$.

{\bf The third invariant.} 
Let now $\rho:S^1\longrightarrow \SO_{q_m}$ be a loop representing the nontrivial element of \(\pi_1(\SO_{q_m})\); for example,
\[
\rho(\theta)=\operatorname{diag}(A(\theta),I_{q_m-2}),
\qquad
A(\theta)=
\begin{pmatrix}
\cos\theta&-\sin\theta\\
\sin\theta&\cos\theta
\end{pmatrix}.
\]
Write $\eta(f)=(\epsilon_1,\ldots,\epsilon_m)\in(\mathbb Z_2)^m$, choosing representatives $\epsilon_i\in\{0,1\}$, and set
\[
r_{\eta(f)}(x_1,\ldots,x_m)
=
\rho(\epsilon_1x_1+\cdots+\epsilon_mx_m).
\]
Then \(r_{\eta(f)}\) induces the homomorphism 
$$
(r_\eta(f))_*=\eta(f):\pi_1(\bT^m)\to\pi_1(SO_{q_m}).
$$
Since \(\SO_{q_m}\)
is a group, we may form
\[
g=r_{\eta(f)}^{-1}\cD f:\bT^m\to SO_{q_m}.
\]
The induced homomorphism
\[
g_*:\pi_1(\bT^m)\longrightarrow\pi_1(\SO_{q_m})
\]
is trivial by construction.
Indeed, pointwise multiplication of maps into a topological group induces addition on $\pi_1$, so that
$$
g_* = (\cD f)_* - (r_{\eta(f)})_* = \eta(f)-\eta(f)=0.
$$ 
Hence, $g$ lifts through the universal covering $p:\operatorname{Spin}_{q_m}\to SO_{q_m}$, that is, there exists $\widetilde g:\bT^m\longrightarrow \operatorname{Spin}_{q_m}$ such that $g=p\circ \widetilde g$.

Since the deck transformation group of $p$ is $\mathbb Z_2$, the two possible lifts are $\widetilde g$ and $-\widetilde g$. Since $\operatorname{Spin}_{q_m}$ is connected, multiplication by $-1$ is homotopic to the identity, and therefore these two lifts determine the same unbased homotopy class.
Conversely, suppose that $p\circ \widetilde g_1 \simeq p\circ \widetilde g_2$ and let $H:\bT^m\times[0,1]\longrightarrow SO_{q_m}$ be such a homotopy. 
Since $\bT^m\times[0,1]$ deformation retracts onto $\bT^m\times\{0\}$ and $H|_{\bT^m\times\{0\}}=p\circ\widetilde g_1$,
the induced homomorphism
\[
H_*:\pi_1(\bT^m\times[0,1])\longrightarrow \pi_1(SO_{q_m})
\]
is trivial. 
Hence, the lifting criterion gives a lift $\widetilde H:\bT^m\times[0,1]\longrightarrow \operatorname{Spin}_{q_m}$ such that $\widetilde H(\,\cdot\,,0)=\widetilde g_1$.
At $t=1$, the map $\widetilde H(\,\cdot\,,1)$ is a lift of
$p\circ\widetilde g_2$, and therefore it is equal to either
$\widetilde g_2$ or $-\widetilde g_2$. Since these two maps are homotopic,
it follows that $\widetilde g_1\simeq \widetilde g_2$.

Consequently, after $\eta$ and the representatives $r_\eta$ have been fixed, the homotopy classes of maps $\bT^m\to SO_{q_m}$ inducing $\eta$ are naturally identified with $[T^m,\operatorname{Spin}_{q_m}]$.

Let now $\sigma\in H^3(\operatorname{Spin}_{q_m};\mathbb Z)\simeq\mathbb Z$ be a generator. 
We define
\[
z(f)=\widetilde g^*\sigma\in H^3(T^m;\mathbb Z).
\]
Notice that $z(f)$ does not depend on which lift has been chosen for $g$ since they are homotopic to each other.

We claim that, in the dimensions under consideration, this is the only remaining invariant.
Indeed, the relevant homotopy groups are $\pi_1(\Spin_{q_m})=\pi_2(\Spin_{q_m})=0$, $\pi_3(\Spin_{q_m})\cong\mathbb Z$ and, whenever they are needed, $\pi_4(\Spin_{q_m})=\pi_5(\Spin_{q_m})=0$.

We can now argue directly using obstruction theory on the standard CW structure of the torus. 
Since \(\operatorname{Spin}_{q_m}\) is \(2\)-connected, every map on the \(2\)-skeleton of \(\bT^m\) is homotopic to the constant map. 
After making it constant there, the restriction of the map to each \(3\)-cell determines an element of $\pi_3(\operatorname{Spin}_{q_m})\simeq \mathbb Z$.
Thus, a map on the \(3\)-skeleton determines a cellular \(3\)-cochain $c\in C^3(\bT^m;\mathbb Z)$.

The obstruction to extending this map over the \(4\)-cells is precisely the cellular coboundary $\delta c\in C^4(\bT^m;\mathbb Z)$.
Hence, the map extends to the \(4\)-skeleton if and only if \(c\) is a cocycle. 
Moreover, the standard obstruction-theory description of homotopies shows
that two such maps determine the same homotopy class if and only if their
associated cocycles differ by a coboundary.
Thus, the extendable maps on the \(3\)-skeleton are classified, up to homotopy, by \(H^3(\bT^m;\mathbb Z)\). 
Since \(\pi_4(\Spin_{q_m})=0\), each extension over the \(4\)-skeleton is unique up to homotopy. For \(m=5\), the additional vanishings $\pi_4(\Spin_{20})=\pi_5(\Spin_{20})=0$
imply that extension over the \(5\)-cells introduces neither a further obstruction nor a further choice up to homotopy.
Consequently, for \(2\leq m\leq 5\),
\[
[\bT^m,\operatorname{Spin}_{q_m}]
\simeq
H^3(\bT^m;\mathbb Z).
\]

    By the Hurewicz theorem, since $\operatorname{Spin}_{q_m}$ is $2$-connected, we have that
\[
H_3(\operatorname{Spin}_{q_m};\mathbb Z)
\simeq
\pi_3(\operatorname{Spin}_{q_m})
\simeq \mathbb Z.
\]
    Hence, by the universal coefficient theorem, $H^3(\operatorname{Spin}_{q_m};\mathbb Z)\simeq \mathbb Z$.
    Choosing $\sigma$ to be a generator, the obstruction-theoretic bijection above is precisely 
    $$
    [\widetilde g]\longmapsto \widetilde g^{\,*}\sigma.
    $$
    Thus, the class in $H^3(\bT^m;\mathbb Z)$ which classifies $\widetilde g$ is exactly $z(f)=\widetilde g^{\,*}\sigma$.

    It follows that, for every fixed $\eta_0\in H^1(\mathbb{T}^m;\mathbb{Z}_2)$, the representative $r_{\eta_0}$ fixed above determines a bijection
\[
\{\,[h]\in[\mathbb{T}^m,SO_{q_m}]\ :\ h_*=\eta_0\,\}
\;\leftrightarrow\;
[\mathbb{T}^m,\mathrm{Spin}_{q_m}]
\;\leftrightarrow\;
H^3(\mathbb{T}^m;\mathbb{Z}),
\,
[h]\longmapsto z(h),
\]
    the first arrow being $[h]\mapsto[\widetilde{r_{\eta_0}^{-1}h}\,]$ and the second $[\tilde g]\mapsto\tilde g^{\,*}\sigma$. 
    In particular, the invariant $z$ is defined on all of $[\mathbb{T}^m,SO_{q_m}]$ once a system of representatives $\{r_\eta\}$ has been fixed.
    Replacing $r_{\eta_0}$ by another representative of the same class $\eta_0$ composes this bijection with a translation of $H^3(\mathbb{T}^m;\mathbb{Z})$; hence no element of the fibre over $\eta_0$ is distinguished, and $z$ itself depends on the choice of the $r_\eta$. 
    What does not depend on that choice is the difference: if $f$ and $g$ are free maps with $\eta(f)=\eta(g)$, then
\[
z(f)-z(g)\in H^3(\mathbb{T}^m;\mathbb{Z})
\]
    is independent of the chosen representatives. 
    When $\eta(f)=0$ we always take $r_0=\mathbf{1}$, and then $z(f)$ is canonical.
\end{proof}

Notice that the result proved above holds also for $m=6$ but fails for $m>6$.

\medskip
The following table shows the sets in which the three invariants of Proposition~\ref{thm:count} take their
values, for $2\le m\le5$. 
Which of these values are actually attained by free maps is discussed below.
 
\begin{center}
\begin{tabular}{c c c c c}
\hline
$m$ & $q_m$ & $\varepsilon$ & $\eta$ & $z$\\
\hline
$2$ & $5$  & $\Z_2$ & $\Z_2^2$  & $\{0\}$\\
$3$ & $9$  & $\Z_2$ & $\Z_2^3$  & $\Z$\\
$4$ & $14$ & $\Z_2$ & $\Z_2^4$ & $\Z^4$\\
$5$ & $20$ & $\Z_2$ & $\Z_2^5$ & $\Z^{10}$\\
\hline
\end{tabular}
\end{center}
 
Thus $m=2$ is the only case in which the set of homotopy invariants is finite: there are exactly $8$ of them.
 
\paragraph{The three invariants.}
For a free map $f$, we write
{\BF \(\varepsilon(f)\)} and {\BF \(\eta(f)\)} for the first two invariants. 
When \(\eta(f)=0\), the invariant {\BF \(z(f)\)} is canonical after taking \(r_0=1\); for arbitrary \(\eta\), \(z(f)\) is defined after fixing the representatives \(r_\eta\).


\medskip
All free maps (except for the one generated by a modification of the ansatz in Section~\ref{sec: obstructions}) exhibited in Section~\ref{sec: tori} have the form
\begin{equation}\label{eq:factorised}
  F(x,z)=R(x)\,v(z),\;
  R(x)=\mathrm{diag}\bigl(A(\langle w_1,x\rangle),\dots,A(\langle w_r,x\rangle),I_s\bigr)
  \in\SO_{q_m},
\end{equation}
with $x=(x^1,\dots,x^{m-1})$, $w_j\in\Z^{m-1}$, and $2r+s=q_m$. 
As pointed out already in Section~\ref{sec: tori}, differentiating \eqref{eq:factorised} and using that $R$ is a homomorphism in each variable implies that 
$$
\cD F(x,z)=R(x)M(z)
$$ 
for a matrix $M$ depending only on $z$. 
For such maps the first $m-1$ components of $\eta$ can be read off directly out the weights.
 
\begin{proposition}[Weight rule]\label{lem:weights}
Let $F$ be as in \eqref{eq:factorised} and free, and let $a_i$ be the $i$-th
coordinate circle of $\bT^{m-1}\subset \bT^m$. Then
\[
  \eta(F)(a_i)\ =\ \sum_{j=1}^{r} w_{ji} \bmod 2 ,
\]
while $\eta(F)$ on the last circle is the class of the loop $z\mapsto M(z)$, where $\cD F(x,z)=R(x)M(z)$, after Gram–Schmidt retraction and, if necessary, multiplication by the fixed reflection.
\end{proposition}
 
\begin{proof}
    Restricting $\cD F$ to $a_i$ with $z$ fixed, the loop is $x^i\mapsto R(x^i)M(z_0)$, which is a right translate of the loop $R|_{a_i}$; right translation does not change the class in $\pi_1$. 
    That loop is a product of planar rotations turning $w_{ji}$ times in mutually orthogonal $2$-planes, and each full turn lifts to a path from $1$ to $-1$ in $\Spin_{q_m}$; hence its class is $\sum_jw_{ji}\bmod2$.
    On the last circle $R$ is constant, so the loop is $z\mapsto M(z)$ up to left translation.
\end{proof}
\begin{proposition}[Vanishing of $z$ for factorized maps]
    \label{lem:zzero}
    Let $m\geq2$ and $k,s\geq1$ with $k+s=m$ and write $\bT^m=\bT^k\times \bT^s$, with coordinates $x=(x_1,\dots,x_k)$ and $y=(y_1,\dots,y_s)$. 
    Let
\begin{equation}\label{eq:gen-ansatz}
F(x,y)=R(x)\,v(y),\,
R(x)=\mathrm{diag}\big(A(\langle w_1,x\rangle),\dots,A(\langle w_r,x\rangle),I_t\big)\in SO_{q_m},
\end{equation}
    with $w_j\in\mathbb{Z}^k$, $2r+t=q_m$, and $v:\bT^s\to\mathbb{R}^{q_m}$ smooth, and assume $F\in\mathrm{Free}^\infty(\bT^m,\mathbb{R}^{q_m})$. 
    Denote by $\pi_s:\bT^m\to \bT^s$ the projection onto the second factor. 
    Then, with respect to the representatives $r_\eta$ fixed above, 
    \[
z(F)\in\pi_s^*\,H^3(\bT^s;\mathbb{Z})\subset H^3(\bT^m;\mathbb{Z}).
\]
    In particular, since $H^3(\bT^s;\mathbb{Z})=0$ for $s\leq2$,
\[
z(F)=0\qquad\text{whenever}\qquad s\leq 2 .
\]
\end{proposition}

\begin{proof}
    After right-multiplication by a fixed reflection, if necessary, and the Gram--Schmidt retraction, which satisfies $\mathrm{GS}(QA)=Q\,\mathrm{GS}(A)$ for $Q$ orthogonal, we may assume
\[
\cD F(x,y)= R(x)\,Q(y),\quad Q:\bT^s\longrightarrow SO_{q_m}.
\]
    Write $\eta(F)=(\epsilon_1,\dots,\epsilon_k,\epsilon_{k+1},\dots,\epsilon_m)$
and split the chosen representative as $r_{\eta(F)}(x,y)=r_x(x)\,r_y(y)$, where
\[
r_x(x)=\rho(\epsilon_1x_1+\cdots+\epsilon_kx_k),\quad
r_y(y)=\rho(\epsilon_{k+1}y_1+\cdots+\epsilon_my_s),
\]
and $\rho$ is the rotation in the first coordinate $2$-plane. 
Since $\rho$ and $R$ take values in the same maximal torus $T\subset SO_{q_m}$, the factor $r_y^{-1}$ commutes with $r_x^{-1}R$, so
\[
r_{\eta(F)}^{-1}\cD F\simeq A(x)B(y),\quad
A(x)=r_x(x)^{-1}R(x)\in T,\quad B(y)=r_y(y)^{-1}Q(y).
\]
By the definition of \(r_x\) and \(r_y\), the maps \(A=r_x^{-1}R\) and \(B=r_y^{-1}Q\) induce the trivial homomorphism on fundamental groups
and so they lift to $\widetilde A:\bT^k\to\mathrm{Spin}_{q_m}$ and $\widetilde B:\bT^s\to\mathrm{Spin}_{q_m}$, and a lift of $r_{\eta(F)}^{-1}\cD F$ is homotopic to the pointwise product $\widetilde g(x,y)=\widetilde A(x)\widetilde B(y)$.
Since $H^1(\mathrm{Spin}_{q_m};\mathbb{Z})=H^2(\mathrm{Spin}_{q_m};\mathbb{Z})=0$,
the K\"unneth formula applied to the multiplication map $$\mu:\Spin_{q_m}\times\Spin_{q_m}\to\Spin_{q_m}$$ 
gives
$
\mu^*\sigma=\sigma\otimes1+1\otimes\sigma
$ 
for a generator $\sigma\in H^3(\mathrm{Spin}_{q_m};\mathbb{Z})\simeq\mathbb{Z}$, so that
\[
z(F)=\widetilde g^{\,*}\sigma
=\pi_k^*\,\widetilde A^{\,*}\sigma+\pi_s^*\,\widetilde B^{\,*}\sigma .
\]
 Now, both \(R(x)\) and \(r_x(x)\) take values in the standard maximal torus of \(SO_{q_m}\), consisting of block-diagonal planar rotations and so  \(A(x)=r_x(x)^{-1}R(x)\) also takes values in this maximal torus.
    Hence, $\widetilde A$ takes values in $p^{-1}(T)$, which is a maximal torus of $\mathrm{Spin}_{q_m}$; the restriction of $\sigma$ to a maximal torus vanishes, since it is represented over $\mathbb{R}$, up to a nonzero constant, by the Cartan $3$-form $\mathrm{tr}\,[g^{-1}dg,g^{-1}dg]\wedge g^{-1}dg$ (see~\cite{CE48}), which vanishes identically on an abelian subgroup, and since the $H^3$ of a torus is torsion free. 
    Thus $\widetilde A^{\,*}\sigma=0$ and $$
    z(F)=\pi_s^*\,\widetilde B^{\,*}\sigma,
    $$ with $\widetilde B^{\,*}\sigma\in H^3(\bT^s;\mathbb{Z})$. 
For $s\leq2$ this group vanishes, and so does $z(F)$.
\end{proof}

\begin{proposition}
    \label{prop: eta(F) non-zero}
    Assume that there exists a  $F\in\Free^\infty(\bT^m,\bR^{q_m})$ with $\eta(F)\neq0$.
    Then, for every $0\neq\eta_0\in H^1(\bT^m,\Z_2)$, there is a  $F_0\in\Free^\infty(\bT^m,\bR^{q_m})$ such that $\eta(F_0)=\eta_0$.
\end{proposition}
\begin{proof}
    The reduction map $GL_m(\mathbb Z)\to GL_m(\mathbb Z_2)$ is surjective, and $GL_m(\mathbb Z_2)$ acts transitively on the nonzero elements of $H^1(\bT^m;\mathbb Z_2)$. 
    Hence, for every $0\neq\eta_0\in H^1(\bT^m;\mathbb Z_2)$, there exists a torus automorphism $\phi:\bT^m\to\bT^m$ such that $\phi^*\eta(F)=\eta_0$.
    
    Now, since $\phi$ is a diffeomorphism, $F_0=F\circ\phi$ is free.
    Moreover, since $\phi$ is induced by an element of $GL_m(\mathbb Z)$, in the standard coordinates its derivative is constant and its second derivatives vanish. 
    Hence, $$\cD(F_0)(x)=\cD F(\phi(x))\,C_\phi$$ for a constant matrix $C_\phi\in GL_{q_m}(\mathbb R)$. 
    Multiplication by this constant matrix does not affect the induced homomorphism on fundamental groups, and therefore $\eta(F_0)=\phi^*\eta(F)=\eta_0$.  
\end{proof}

\paragraph{Can all invariants be reached by free maps?}
Consider the freedom relation ${\cal R}_m\subset J^2(\bT^m,\bR^{q_m})$, whose fiber is $\bR^{q_m}\times GL_{q_m}$, and let $\Gamma^\infty{\cal R}_m$ be the set of smooth section of the bundle ${\cal R}_m\to \bT^m$.
Clearly, by a slight abuse of notation, we have that $Free^\infty(\bT^m,\bR^{q_m})\subset \Gamma^\infty{\cal R}_m$.
As we illustrated at the beginning of the section, $\pi_0(\Gamma^\infty{\cal R}_m)=[\bT^m,GL_{q_m}]$.
It is a natural question to ask whether the map $\pi_0(Free^\infty(\bT^m,\bR^{q_m}))\to \pi_0(\Gamma^\infty{\cal R}_m)$ is surjective.
Notice that this map is surjective in case of free maps on tori in non-critical dimension because of the parametric h-Principle (see~\cite{Gro86}, Section 1.2.1).
Below we give some first partial answer in case of critical dimension.

\paragraph{The case of $\bT^2$.}
Since $H^3(\bT^2,\Z)=0$, free maps on $\bT^2$ can have only two non-trivial homotopy invariants.
For the free embedding $F_2$ defined in Section~\ref{sec: tori} we have that 
$$
\varepsilon(F_2)=\sgn\det\cD F_2=+1
$$ 
and 
$$
\eta(F_2)(a_1) = 1 + 2 = 1\ (\text{mod}\ 2).
$$
In particular, $\eta(F_2)\neq0$.
Hence, by Proposition~\ref{prop: eta(F) non-zero}, there are free maps with any given non-zero invariant in $H^1(\bT^2,\bZ_2)$.
Moreover, let $\psi:(x,y)\mapsto(2x,2y)$.
Then $\psi$ is a four-fold cover of $\bT^2$ and $\psi^*:H^1(\bT^2,\bZ_2)\to H^1(\bT^2,\bZ_2)$ is the zero map.
Since $\psi$ is a local diffeomorphism, $F\circ\psi$ is still free but now $\eta(F\circ\psi)=\psi^*\eta(F)=0$.
Finally, postcomposing with the fixed reflection of \(\mathbb R^5\) reverses \(\varepsilon\) while leaving the normalized \(SO_5\)-class, and so \(\eta(F)\), unchanged.
Hence, the map $\pi_0(Free^\infty(\bT^2,\bR^{5}))\to \pi_0(\Gamma^\infty{\cal R}_2)$ is surjective.


\paragraph{The case of $\bT^3$.}
For the free embedding $F_3$ defined in Section~\ref{sec: tori} we have that $\varepsilon(F_3)=\sgn\det\cD F_3=-1$, and $\eta(F_3)(a_1) = 1\ (\text{mod}\ 2)$ and $\eta(F_3)(a_2) = 1\ (\text{mod}\ 2)$. 
Since $\eta(F_3)\neq0$, by Proposition~\ref{prop: eta(F) non-zero} we can achieve all non-zero elements in $H^1(\bT^3,\bZ_2)$.
By the same argument used above for $\bT^2$, we can also reach the 0 element in $H^1(\bT^3,\bZ_2)$.
On the other hand, by Proposition~\ref{lem:zzero} we have that $z(F_3)=0$.
Notice that this does not mean that necessarily every free map $\bT^3\to\bR^9$ has its third invariant equal to zero, this is simply a feature of our maps due to the nature of our ansatz.
Hence, for $m\geq3$ we are not able to establish whether or not $\pi_0(Free^\infty(\bT^m,\bR^{q_m}))\to \pi_0(\Gamma^\infty{\cal R}_m)$ is surjective.


\paragraph{The cases of $\bT^4$ and $\bT^5$.}
The reader can easily verify that $\eta(F_4)\neq0$ and $\eta(F_5)\neq0$.
Hence, as in case of $\bT^3$, we can obtain all possible values of $\varepsilon$ and $\eta$ but, with our ansatz, we have always $z=0$ and we cannot establish whether or not there are free maps in each connected component of $\pi_0(\Gamma^\infty{\cal R}_m)$, $m=4,5$. 

\medskip
As our final result, we state the following corollary of the results above.

\begin{corollary}
    For $2\leq m\leq5$, the set $\Free^\infty(\bT^m,\bR^{q_m})$ has at least $2^{m+1}$ connected components.
\end{corollary}


\section*{Acknowledgments}
The author warmly thanks M. Gromov and Y. Eliashberg for their valuable comments on earlier versions of the manuscript.
The author was partially supported by NSF grant \# 2308225.

\bibliographystyle{unsrt}
\bibliography{refs.bib} 

@ARTICLE{EG71,
   author  = "Ya. M. Eliashberg and M. Gromov",
   year    = "1971",
   journal = "Math. USSR Izvestija",
   title   = "Removal of singularities of smooth mappings",
   volume  = "5",
   pages   = "615-639"
}

@BOOK{EM02,
   author    = "Y. Eliashberg and N. Mishachev",
   year      = "2002",
   title     = "Introduction to the $h$-Principle",
   publisher = "AMS",
   series    = "GSM",
   volume    = "48"
}

@ARTICLE{GR70,
   author  = "M. Gromov and V.A. Rokhlin",
   year    = "1970",
   journal = "Russian Math. Surveys",
   title   = "Immersions and embeddings in {R}iemannian geometry",
   volume  = "25",
   pages   = "1-57"
}

@BOOK{Gro86,
   author    = "M. Gromov",
   year      = "1986",
   title     = "Partial {D}ifferential {R}elations",
   publisher = "Springer Verlag",
}

@ARTICLE{Nas56,
   author  = "J. Nash",
   year    = "1956",
   title   = "The imbedding problem for {R}iemannian manifolds",
   journal = "Ann. of Math.",
   volume  = "63:1",
   pages   = "20-63"
}

@article{Gro17,
  title={Geometric, algebraic, and analytic descendants of {N}ash isometric embedding theorems},
  author={Gromov, M.},
  journal={Bulletin of the American Mathematical Society},
  volume={54},
  number={2},
  pages={173--245},
  year={2017}
}

@article{Cos90,
  title={On closed twisted curves},
  author={Rodriguez Costa, S.I.},
  journal={Proceedings of the American Mathematical Society},
  volume={109},
  number={1},
  pages={205--214},
  year={1990}
}

@article{CE48,
  author  = {Chevalley, C. and Eilenberg, S.},
  title   = {Cohomology theory of {L}ie groups and {L}ie algebras},
  journal = {Transactions of the American Mathematical Society},
  volume  = {63},
  number  = {1},
  pages   = {85--124},
  year    = {1948},
  doi     = {10.1090/S0002-9947-1948-0024908-8},
  mrnumber = {MR0024908},
}

\end{document}